\RequirePackage{fix-cm}
\documentclass[smallextended]{svjour3}       
\smartqed  
\usepackage{graphicx}
\usepackage{amssymb,amsfonts,amsmath}
\usepackage{amsfonts,amsmath}
\usepackage{tikz}
\usepackage{algorithm}
\usepackage{algorithmic}
\usepackage{setspace}

\newcommand{\abs}[1]{\lvert#1\rvert}
\newcommand{\norm}[1]{\lVert#1\rVert}

\newcommand{\xdif}{{\rm d}}
\newcommand{\dx}{\; \xdif x}

\newcommand{\cD}{\ensuremath{\mathcal{D}}}

\newcommand{\cN}{\ensuremath{\mathcal{N}}}

\newcommand{\cU}{\ensuremath{U}}

\newcommand{\rmy}{\ensuremath{\mathrm{y}}}
\newcommand{\rmu}{\ensuremath{\mathrm{u}}}
\newcommand{\rmp}{\ensuremath{\mathrm{p}}}
\newcommand{\rmz}{\ensuremath{\mathrm{z}}}
\newcommand{\rmU}{\ensuremath{\mathrm{U}}}
\newcommand{\rmD}{\ensuremath{\mathrm{D}}}
\newcommand{\rmC}{\ensuremath{\mathrm{C}}}
\newcommand{\rmA}{\ensuremath{\mathrm{A}}}
\newcommand{\rmF}{\ensuremath{\mathrm{F}}}
\newcommand{\rmM}{\ensuremath{\mathrm{M}}}

\newcommand{\xR}{\ensuremath{\mathbb{R}}}
\newcommand{\xN}{\ensuremath{\mathbb{N}}}
\newcommand{\xK}{\ensuremath{\mathbb{K}}}
\newcommand{\xI}{\ensuremath{\mathbb{I}}}

\newcommand{\ru}{\ensuremath{r_u}}

\newcommand{\train}{\mathrm{train}}

\newcommand{\muref}{\ensuremath{\mu^\mathrm{ref}}}

\newcommand{\alphaa}{\ensuremath{\alpha(\mu)}}
\newcommand{\alphaaLB}{\ensuremath{\alpha_{\rm LB}(\mu)}}

\newcommand{\gammab}{\ensuremath{\gamma_b}}
\newcommand{\gammac}{\ensuremath{\gamma_c}}

\newcommand{\tolmin}{\mathrm{tol, min}}
\newcommand{\rmmax}{{\rm max}}

\newcommand{\rme}{\mathrm{e}}

\newcommand{\ryk}{\ensuremath{r_y^k}}
\newcommand{\rpk}{\ensuremath{r_p^k}}

\newcommand{\rykt}{\ensuremath{\tilde{r}_y^k}}
\newcommand{\rpkt}{\ensuremath{\tilde{r}_p^k}}
\newcommand{\rukt}{\ensuremath{\tilde{r}_u^k}}

\newcommand{\eyk}{\ensuremath{e_y^k}}
\newcommand{\epk}{\ensuremath{e_p^k}}
\newcommand{\euk}{\ensuremath{e_u^k}}
\newcommand{\eykm}{\ensuremath{e_y^{k-1}}}
\newcommand{\epkp}{\ensuremath{e_p^{k+1}}}

\newcommand{\yNoptk}{\ensuremath{y_N^{*,k}}}
\newcommand{\yNoptkm}{\ensuremath{y_N^{*,k-1}}}

\newcommand{\pNoptk}{\ensuremath{p_N^{*,k}}}
\newcommand{\pNoptkp}{\ensuremath{p_N^{*,k+1}}}

\newcommand{\uNoptk}{\ensuremath{u_N^{*,k}}}

\newcommand{\refeq}[1]{(\ref{#1})}

\DeclareMathOperator{\spn}{span}
\DeclareMathOperator*{\argmin}{arg\,min}

\usepackage{color}
\usepackage{xcolor}

\newcommand{\cbb}[1]{{\color{black}#1}}

\begin{document}

\title{Reduced basis approximation and a~posteriori error bounds for 4D-Var data assimilation\thanks{This work was supported by the Excellence Initiative of the German federal and state governments and the German Research Foundation through Grant GSC 111.}}


\titlerunning{{RB approximation and error bounds for 4D-Var data assimilation}}

\author{Mark K\"{a}rcher \and S\'{e}bastien Boyaval \and
Martin A. Grepl \and Karen Veroy}


\institute{Mark K\"{a}rcher \at
		Aachen Institute for Advanced Study in Computational
		Engineering Science (AICES), RWTH Aachen University,
		Schinkelstra{\ss}e 2, 52062 Aachen, Germany \\
		\email{kaercher@aices.rwth-aachen.de}    
	\and
		S\'{e}bastien Boyaval \at
		Laboratoire d'hydraulique Saint-Venant
		(Ecole des Ponts ParisTech -- EDF R\& D -- CEREMA)
		Universit\'{e} Paris-Est, 6 quai Watier, 78401 Chatou Cedex, France 
		and INRIA Paris (team Matherials)
		\\ \email{sebastien.boyaval@enpc.fr}
	\and
	Martin A. Grepl \at
	Numerical Mathematics (IGPM), RWTH Aachen University, Templergraben 55, 52056 Aachen, Germany \\
	\email{grepl@rwth-aachen.de}
	\and
	Karen Veroy \at
	Aachen Institute for Advanced Study in Computational
	Engineering Science (AICES) \\ and Faculty of Civil Engineering, RWTH Aachen University,
	Schinkelstra{\ss}e 2, 52062 Aachen \\
	\email{veroy@aices.rwth-aachen.de}
}

\date{Received: date / Accepted: date}

\maketitle

\begin{abstract}
We propose a certified reduced basis approach for the strong- and weak-constraint four-dimensional variational (4D-Var) data assimilation problem for a parametrized PDE model. While the standard strong-constraint 4D-Var approach uses the given observational data to estimate only the unknown initial condition of the model, the weak-constraint 4D-Var formulation additionally provides an estimate for the model error and thus can deal with imperfect models. Since the model error is a distributed function in both space and time, the 4D-Var formulation leads to a large-scale optimization problem for every given parameter instance of the PDE model. To solve the problem efficiently, various reduced order approaches have therefore been proposed in the recent past. Here, we employ the reduced basis method to generate reduced order approximations for the state, adjoint, initial condition, and model error. Our main contribution is the development of efficiently computable \textit{a~posteriori} upper bounds for the error of the reduced basis approximation with respect to the underlying high-dimensional 4D-Var problem. Numerical results are conducted to test the validity of our approach.


\end{abstract}



\section{Introduction}
\label{sec:intro}


The goal of four-dimensional variational (4D-Var) data assimilation is to estimate unknown control variables of a dynamical system --- classically the initial condition of the system --- that provide the best fit of the system outputs with observation data over a specific time interval~\cite{Courtier1997,DT1986,Lorenc1981,Lorenc1986,Sasaki1970}. The use of 4D-Var data assimilation is prevalent in oceanography~\cite{Bennett1993} and meteorology~\cite{Lynch2015}, where the dynamical system is described by partial differential equations (PDEs); see the recent texts~\cite{LSZ2015,RC2015} and references therein for variational data assimilation in general.  

We consider two variants of the 4D-Var problem.  In the traditional {\it strong-constraint} 4D-Var formulation, the model is assumed to be ``perfect'' and only the initial conditions serve as the (unknown) control variable. The {\it weak-constraint} 4D-Var formulation additionally accounts for an imperfect model in the traditional formulation by introducing and finding a forcing term to account for the model error.  In the weak-constraint case, the unknown initial condition \textit{and} unknown model-error forcing term thus serve as control variables; for various weak-constraint formulations see e.g.~\cite{Tremolet2006}.

The 4D-Var problem is usually cast as an optimization problem and has very close connections to optimal control theory~\cite{VH2006}. A cost functional is introduced consisting of two terms in the classical strong-constraint formulation: the first term penalizes the misfit between the (unknown) initial condition and its prior background information and the second term penalizes the distance between the predicted system outputs and the observation data. In the weak-constraint case, another term is added which penalizes the model-error forcing. The optimal estimate of the initial condition is then found by minimizing the cost functional subject to the governing equations of the dynamical system, i.e., the PDE. After discretization of the PDE using classical techniques such as finite elements or volumes, the 4D-Var problem results in a large-scale optimization problem which is typically very expensive to solve due to the high-dimensional state and control variable spaces and the associated computation of the cost functional, gradient, and possibly Hessian. Note that in the discretized weak-constraint formulation, the model-error forcing is also assumed to be spatially distributed and thus has approximately the same dimension as the state and initial condition. To lower the tremendous computational cost for solving the problem, an incremental approach has been proposed in~\cite{CTH1994}.

Another possibility to speed-up the solution process are reduced-order approaches which have been proposed successfully for the strong-constraint 4D-Var formulation in, for example, \cite{CZN+2007,DN2008,DAS2010,HK2006,RDB+2005,VH2006}. There are two kinds of 4D-Var reduced-order approaches in the literature: In the first approach~\cite{HK2006,RDB+2005,VH2006}, a reduced basis space is introduced, e.g. using empirical orthogonal functions, only for the control variable (initial condition). By limiting the search space to the reduced space, the optimization cost per iteration decreases and the convergence improves (at least during the first few iterations). In the second approach~\cite{CZN+2007,DN2008,DAS2010}, a reduced-order model for the system dynamics using proper orthogonal decomposition (POD) is additionally introduced. This leads to an additional speed-up and significant overall computational savings compared to reducing only the control space.  All of these approaches also consider adapting the basis during the optimization. However, to the best of our knowledge, \textit{a~posteriori} error bounds to assess the sub-optimality of the reduced-order 4D-Var solutions have not yet been developed.


In this paper, we develop efficiently evaluable \textit{a~posteriori} error bounds for reduced order solutions of the strong- and weak-constraint 4D-Var data assimilation problem. We consider the standard quadratic 4D-Var cost functional constrained by parametrized linear parabolic PDEs involving noisy observations in time. Our final goal is not only to recover the ``usual'' 4D-Var control variables, i.e. the initial condition and model-error forcing, but also the model parameters. \cbb{A preliminary improvement of the model itself before estimating the state can result in an improved state estimate, see e.g.~the application in~\cite{Habert2016}.} We thus obtain a bilevel optimization problem where the outer optimization stage is performed over the model parameters after an inner optimization stage identical to the standard 4D-Var setting, i.e., an optimization over control variables for given fixed model parameters. In this paper, we focus mainly on the inner optimization stage and propose \textit{a~posteriori} error bounds for the control variable.  Our main contributions are as follows:
\begin{itemize}
\item In Section~\ref{sec:strong}, we consider the strong-constraint 4D-Var formulation. We employ the reduced basis method to generate reduced order approximations for the solution of the parametrized 4D-Var problem, i.e., the state, adjoint, and control variables (i.e., the initial condition). We then propose an \textit{a~posteriori} error bound for the control variable that allows us to assess the error between the reduced-order 4D-Var solution and 4D-Var solution of the underlying high-dimensional FE approximation. 
\item In Section~\ref{sec:weak}, we extend the reduced basis approximation and \textit{a~posteriori} error estimation procedure from the strong- to the weak-constraint case. For simplicity of exposition, we consider the model-error forcing as the only unknown control variable in this section.
\item In Section~\ref{sec:comb}, we combine the results from the two previous sections and consider problems with  unknown initial condition {\it and} model-error forcing. 
\end{itemize}
\cbb{With the assumption of affine parameter dependence, the reduced-order 4D-Var problems and the \textit{a~posteriori} error bounds can be efficiently evaluated using an offline-online computational decomposition. Problems involving material parameters often naturally satisfy an affine parameter dependence, and even geometric parameters can often be treated after introducing suitable affine mappings onto a reference domain~\cite{RHP2008}. Furthermore, the dimension reduction as well as the \textit{a posteriori} error bound formulation presented in this paper still hold even for non-affine problems.  However, for non-affine problems the computations can no longer be decomposed into offline-online stages, and the online computational efficiency thus suffers. To address this issue, the non-affine case can be treated using the Empirical Interpolation Method (EIM) which replaces the non-affine terms using an affine approximation and thus allows to regain the online-computational efficiency; we refer the interested reader e.g.\ to \cite{BMN+2004,GMN+2007,MNP+2007}.}

We present numerical results for the strong- and weak-constraint setting in Section~\ref{sec:results}. We consider the dispersion of a pollutant governed by a convection-diffusion equation with a Taylor-Green vortex velocity field. Our goal is to recover the initial condition (in the strong-constraint case) or the model-error forcing (in the weak-constraint case) given noisy measurements of the pollutant concentration at five spatial locations over time. 

\cbb{We note that there is a close connection between the 4D-Var problem formulation and optimal control and that \textit{a posteriori} error bounds for reduced order solutions to optimal control problems have been developed previously. However, rigorous \textit{and} efficiently evaluable error bounds have been proposed mainly for elliptic problems~\cite{KG2012,Kaercher2017,NRM+2012}, whereas error bounds for parabolic optimal control problems are either not rigorous~\cite{Dede2010} or not (online-)efficient~\cite{TV2009}. The only exception for parabolic problems is~\cite{KG2014}, which only considers scalar time-dependent controls and is based on a pertubation argument, often resulting in a more conservative error bound~\cite{Kaercher2017}.}

Finally, we note that the reduced basis method has already been used in a para\-meterized-background data-weak approach to variational data assimilation in~\cite{MPP+2015,MPP+2015a}. However, this previous work considers the elliptic case and presents a relaxation of the 3D-Var setting, whereas we consider the time-dependent case using the classical 4D-Var formulation. Before introducing some preliminary definitions and assumptions in the following section, we do note that although we consider the 4D-Var problem here, our approach directly applies to the 3D-Var setting since the two are formally similar~\cite{Lynch2015}.

%
%
%
%
%

\section{Preliminaries}
\label{sec:pre}

In this section, we introduce the necessary ingredients and definitions for the subsequent discussion. The 4D-Var problem is usually cast in a fully discrete setting; we thus directly consider a spatial finite element (FE) and temporal finite difference (FD) discretization using the weak variational formulation. \cbb{We summarize the continuous formulation of the 4D-Var problem in Appendix~\ref{sec:cont}.}

Let $Y_\rme$ with $H^1_0(\Omega) \subset Y_\rme \subset H^1(\Omega)$ be a Hilbert space of functions over the bounded Lipschitz domain $\Omega \subset \xR^d$, $d \in \xN$, with boundary $\Gamma$. The inner product and induced norm associated with $Y_\rme$ are given by $(\cdot,\cdot)_Y$ and $\norm{\cdot}_Y = \sqrt{(\cdot,\cdot)_Y}$, respectively. We assume that the norm $\norm{\cdot}_Y$ is equivalent to the $H^1(\Omega)$-norm and denote the dual space of $Y_\rme$ by $Y_\rme'$. We also introduce the Hilbert space for the control, $U_\rme = L^2(\Omega)$, together with its inner product $(\cdot,\cdot)_U$, induced norm $\norm{\cdot}_U = \sqrt{(\cdot,\cdot)_U}$, and associated dual space $U_\rme'$. Furthermore, let $\cD \subset \xR^P$ be a prescribed $P$-dimensional compact set in which our $P$-tuple input parameter $\mu = (\mu_1,\ldots,\mu_{P})$ resides.

We divide the time interval $[0,T]$ with fixed final time $T$ into $K$ subintervals of equal length $\tau = \frac{T}{K}$ and define $t^k = k \, \tau, \ 0 \le k \le K$, and $\xK = \{ 1, \dots, K \}$. We also introduce two conforming finite element approximation spaces $Y \subset Y_\rme$ and $U \subset U_\rme$ of typically large dimension $\cN_Y = \dim(Y)$ and $\cN_\cU = \dim(\cU)$; note that $Y$ and $U$ shall inherit the inner product and norm from $Y_\rme$ and $U_\rme$, respectively.  We shall assume that the spaces $Y,\cU$ and the number of timesteps $K$ are large enough -- i.e.\ $Y$ and $U$ are sufficiently rich and the time-discretization sufficiently fine -- such that the FE-FD approximation guarantees a desired accuracy over the whole parameter domain $\cD$.

We next introduce the (for the sake of simplicity) parameter-independent bilinear forms $m(w,v) = (w,v)_{L^2(\Omega)}$ for all $w,v \in L^2(\Omega)$ and $b(\cdot,\cdot): U \times Y \to \xR$. We assume that $b(\cdot,\cdot)$ is continuous, i.e.
\begin{equation}
\label{b_cont}
	\gammab = \sup_{w \in U \setminus \{0\}} \sup_{v \in Y \setminus \{0\}} \frac{b(w,v)}{\norm{w}_{U} \norm{v}_{Y}} < \infty.
\end{equation}
We also introduce the parameter-dependent bilinear form $a(\cdot,\cdot;\mu): Y \times Y \to \xR$, which we assume to be continuous,
coercive,
\begin{equation}
	\label{a_coer}
	\alphaa = \inf_{v \in Y \setminus \{0\}} \frac{a(v,v;\mu)}{\norm{v}_Y^2}
	\ge \underline{\alpha} > 0 \quad \forall \mu \in \cD,
\end{equation}
and affinely parameter-dependent,
\begin{equation}
	\label{aff_par_a}
	a(w,v;\mu) = \sum_{q=1}^{Q_a} \Theta_a^q(\mu) \, a^q(w,v) \quad \forall w,v \in Y,
	\quad \forall \mu \in \cD,
\end{equation}
for some (preferably) small integer $Q_a$. Here, the coefficient functions $\Theta_a^q: \cD \rightarrow \xR$ are continuous and depend on $\mu$, but the continuous bilinear forms $a^q: Y \times Y \rightarrow \xR$ do \textit{not} depend on $\mu$. 

We also require the continuous linear functional $f(\cdot): Y \to \xR$ and the continuous and linear (observation) operator $C: Y \to D$, where $D$ is a suitable Hilbert space of observations with inner product $(\cdot,\cdot)_D$ and norm $\norm{\cdot}_D$. Although a more general setting is possible, we consider here the observation space $D = \xR^l$ and the observation operator given by $C \phi = (h_1(\phi), \dots, h_{\ell}(\phi))^T$, where $h_i \in Y'$ are linear output functionals. The continuity constant of the operator $C$ is given by
\begin{equation}
	\gammac = \sup_{v \in Y \setminus \{0\}} \frac{\norm{C v}_{D}}{\norm{v}_Y}.
\end{equation}

For the development of the \textit{a~posteriori} error bounds we assume that we have access to a positive lower bound  $\alphaaLB: \cD \to \xR_{+}$ for the coercivity constant $\alpha(\mu)$ defined in \eqref{a_coer} such that
\begin{equation}
\label{alphaLB}
0 < \underline{\alpha} \leq \alphaaLB \leq \alphaa \quad \forall \mu \in \mathcal{D}.
\end{equation}
We note that $\alphaaLB$ is used in the \textit{a~posteriori} error bound formulation to replace the actual coercivity constants. Whereas the constants $\gammab$ and $\gammac$ are parameter-independent and can thus be computed once offline, we require that the coercivity lower can be efficiently evaluated online, i.e., the computational cost is independent of the FE dimension $\cN$. Various recipes exist to obtain such bounds~\cite{HRS+2007,RHP2008}.

\section{Strong-constraint 4D-Var}
\label{sec:strong}

In this section, we consider the strong-constraint 4D-Var data assimilation problem. The extension to the weak-constraint case is considered in Section \ref{sec:weak}.

\subsection{Problem statement}
\label{sec:ps_strong}

For a given parameter $\mu \in \cD$, the classical 4D-Var problem can be stated as the minimization problem
\begin{gather}
\begin{aligned}
	\label{fe_strong}
	& \min_{y \in Y^K, \, u \in U} J(y,u;\mu)  \quad \text{s.t.} \quad y \in Y^K \quad\text{solves} \\ & m(y^k,v) + \tau \,
	a(y^k,v;\mu) = m(y^{k-1},v) + \tau f(v) \quad \forall v \in Y,
	\ \forall k \in \xK,
\end{aligned}
\end{gather}
with initial condition $m(y^0,v) = m(u,v)$ for all $v \in Y,$ and cost functional $J(\cdot,\cdot;\mu): Y^K \times U \to \xR$ given by 
\begin{equation}
\label{cost_fe_strong}
 J(y,u;\mu) = \frac{1}{2} \norm{u - u_d}_{U}^2	+ \frac{\tau}{2} \sum_{k=1}^K \norm{C y^k - z_d^k}^2_{D}.
\end{equation}
Here, $u_d \in U$ is the background state (also referred to as the prior), i.e., the best estimate of the true initial condition $u \in U$ prior to measurements being available, and $z_d^k \in D$, $k \in \xK$, is the given data, e.g., observed outputs. The first term in the cost functional penalizes the deviation of the initial condition from the background state, the second term penalizes the deviation of the predicted outputs from the given data/observed outputs. The relative weight of both terms is affected by the choice of the $(\cdot,\cdot)_U$ and $(\cdot,\cdot)_D$ inner products. Note that we use $u$ for the unknown control/initial condition to signify the similarity to optimal control and the notation $J(\cdot,\cdot;\mu)$ to indicate the implicit dependence of the cost functional $J$ on the parameter $\mu$ through the state $y$. However, to simplify the notation we often do not explicitly state the dependence of the state and control on the parameter, i.e., we use $y^k$ and $u$ instead of $y^k(\mu)$ and $u(\mu)$, respectively.

\cbb{We would like to point out that the first term in \eqref{cost_fe_strong} represents a Tichonov regularization of the cost functional and that the regularization parameter is ``hidden'' in the choice of the inner product. We refer to~\cite{Engl1996} for regularization of inverse problems in general and to~\cite{Puel2009} for Tichonov regularization in data assimilation. Furthermore, we note that the choice of the norm for the data misfit term depends on the characteristics of the noise and is inspired by Gaussian noise in this paper. Different noise characteristics may require a different choice of norm; we refer e.g.\ to \cite{Rao2017} for a discussion using $L_1$ and Huber norms instead of the $L_2$ norm. The approach presented in the following is restricted to the case of Gaussian noise.}

Employing a Lagrangian approach, we obtain the associated necessary, and in our setting sufficient, first-order optimality conditions: Given $\mu \in \cD$, the optimal solution $(y^*,p^*,u^*) \in Y^K \times Y^K \times U$ satisfies
\begin{subequations}
\label{OS_fe_strong}
\begin{align}
	m(y^{*,k} - y^{*,k-1},\phi) + \tau \,  a ( y^{*,k},\phi;\mu) & = \tau \, f(\phi)
	&& \forall \phi \in Y, \ \forall k \in \xK, \label{OS_fe_strong:1} \\ 
	m(y^{*,0}, \phi) & = m(u^*,\phi) && \forall \phi \in Y, \label{OS_fe_strong:2} \\
	m(\varphi, p^{*,k} - p^{*,k+1}) + \tau \,  a ( \varphi,p^{*,k};\mu)  &=
	 \tau \,  (z_d^k - C y^{*,k}, C \varphi)_{D} \hspace{-8ex} \notag
	\\ & &&  \forall \varphi \in Y, \ \forall k \in \xK,  \label{OS_fe_strong:3} \\
	(u^{*} - u_d,\psi)_U  - m(\psi, p^{*,1}) &= 0 &&
	\forall \psi \in U, \label{OS_fe_strong:4}
\end{align}
\end{subequations}
where the final condition of the adjoint is given by $p^{*,K+1} = 0$. \cbb{Concerning the existence and uniqueness of the 4D-Var problem specifically and of saddle point problems in general we refer to~\cite{Broecker2017} and~\cite{BGL2005a}.}

\subsubsection{Algebraic Formulation}
\label{ssec:alg_form}


The 4D-Var problem is usually stated using an algebraic formulation~\cite{Ide1997}. We thus briefly outline the algebraic equivalent of~\refeq{fe_strong} by introducing a basis for the finite element spaces $Y$  and $U$ such that $Y = \spn \{\phi_i^y, \, i = 1, \ldots, \cN_Y\}$ and $U = \spn \{\phi_i^u, \, i = 1, \ldots, \cN_U\}$, respectively.  We express the state, adjoint, and control, respectively, as 
$$y^k = \textstyle \sum\limits_{i = 1}^{\cN_Y} y^k_i \phi_i^y, \qquad p^k = \textstyle \sum\limits_{i = 1}^{\cN_Y} y^k_i \phi_i^y, \qquad u = \textstyle \sum\limits_{i = 1}^{\cN_U} u_i \phi_i^u,$$
\noindent and denote the corresponding coefficient vectors by $\rmy^k = [y_1^k, \ldots, y_{\cN_Y}^k]^T \in \xR^{\cN_Y}$, $\rmp^k = [p_1^k, \ldots, p_{\cN_Y}^k]^T \in \xR^{\cN_Y}$, and $\rmu = [u_1, \ldots, u_{\cN_U}]^T \in \xR^{\cN_U}.$  We thus obtain the algebraic formulation of the classical 4D-Var minimization problem
\begin{gather}
\begin{aligned}
	\label{fe_strong_alg}
	& \min J(\rmy,\rmu;\mu) = \frac{1}{2} (\rmu
	- \rmu_b)^T \rmU  (\rmu - \rmu_b)
	+ \frac{\tau}{2} \sum_{k=1}^K (\rmC \rmy^k - \rmz_d^k)^T \rmD
	(\rmC \rmy^k - \rmz_d^k), \\
	&  \text{s.t.} \quad \rmy^k \in \xR^{\cN_Y} \quad\text{solves}\quad
	\rmM \rmy^k + \tau \, \rmA(\mu) \rmy^k = \rmM \rmy^{k-1}
	+ \tau \rmF \quad \forall k \in \xK, \\ & \text{with initial condition}\quad
	\rmM \rmy^0 = \rmM_u \rmu.
\end{aligned}
\end{gather}
Here, $\rmM\in \xR^{\cN_Y \times \cN_Y}$, $\rmA(\mu) \in \xR^{\cN_Y \times \cN_Y}$, $\rmF\in \xR^{\cN_Y}$, and $\rmC\in \xR^{\ell \times \cN_Y}$ are the usual finite element mass matrix, stiffness matrix, load vector, and state-to-output matrix with entries $\rmM_{ij} = m(\phi_j^y,\phi_i^y)$, $\rmA_{ij}(\mu) = a(\phi_j^y,\phi_i^y;\mu)$, $\rmF_{i} = f(\phi_i^y)$, and $\rmC_{ij} = h_i(\phi_j^y)$, respectively. The matrix $\rmM_u \in \xR^{\cN_Y \times \cN_U}$ is given by $(\rmM_u)_{ij} = m(\phi_j^u,\phi_i^y)$. Furthermore, the matrices $\rmU \in \xR^{\cN_Y \times \cN_Y}$ with entries $\rmU_{ij} = (\phi_j^y,\phi_i^y)_U$ and $\rmD \in \xR^{\ell \times \ell}$ with entries $\rmD_{ij} = (e_j,e_i)_D$ can be identified as the inverses of the background and observation error covariance matrices, respectively. Here, $e_i$ denotes the $i$th unit vector in $\xR^{\ell}$.

The derivation and algebraic formulation of the optimality system~\refeq{OS_fe_strong} is standard and thus omitted for brevity. Further, in our problem setting the first-discretize-then-optimize and first-optimize-then-discretize strategies lead to the same algebraic formulation of the first-order optimality system. For more details on these two approaches, we refer to~\cite{HPU+2009} \cbb{and for time-dependent problems specifically to \cite{SW2013}.}

\subsection{Reduced basis approximation}
\label{ssec:rb_strong}

We first assume that we are given the reduced basis spaces $Y_N \subset Y$ for the state and adjoint, and $U_N^0 \subset U$ for the control. Here, $1 \le N \le N_{\rmmax}$ is the number of iterations of the POD-Greedy sampling procedure to construct the spaces $Y_N$ and $U_N^0$ discussed in Section~\ref{ssec:greedy_weak}. Note that the dimensions $N_Y(N) := \dim(Y_N)$ and $N_U^0(N) := \dim(U_N^0)$ of the reduced basis spaces depend on $N$ but are in general not equal to $N$. Furthermore, the basis functions of $Y_N$ and $U_N^0$ are orthogonalized with respect to the $(\cdot,\cdot)_Y$ and $(\cdot,\cdot)_U$ inner product, respectively.


We next replace the finite element approximation of the PDE constraint in the 4D-Var problem statement~\eqref{fe_strong} with its reduced basis approximation. For a given parameter $\mu \in \cD$, the reduced-order 4D-Var data assimilation problem can thus be stated as
\begin{gather}
\begin{aligned}
	\label{rb_strong}
	& \min_{y_N \in Y_N^K, \, u_N \in U_N^0} J(y_N,u_N;\mu) \quad \text{s.t.} \quad y_N \in Y_N^K \quad\text{solves} \\ & m(y_N^k,v)
	+ \tau \, a(y_N^k,v;\mu) = m(y_N^{k-1},v) + \tau f(v) \quad \forall v \in
	Y_N, \ \forall k \in \xK,
\end{aligned}
\end{gather}
with initial condition $m(y_N^0,v) = m(u_N,v)$ for all $v \in Y_N$.

We can again employ a Lagrangian approach to obtain the reduced-order optimality system: 
Given any $\mu \in \cD$, the optimal solution $(y_N^*,p_N^*,u_N^*) \in Y_N^K \times Y_N^K \times U_N^0$ satisfies
\begin{subequations}
\label{OS_rb_strong}
\begin{align}
	m(y_N^{*,k} - y_N^{*,k-1},\phi) + \tau \,  a (y_N^{*,k},\phi;\mu) & = \tau \, f(\phi)
	&& \forall \phi \in Y_N, \ \forall k \in \xK, \label{OS_rb_strong:1} \\ 
	m(y_N^{*,0}, \phi) & = m(u_N^*,\phi) && \forall \phi \in Y_N, \label{OS_rb_strong:2} \\
 m(\varphi, p_N^{*,k} - p_N^{*,k+1}) + \tau \,  a (\varphi,p_N^{*,k};\mu) &=
	\tau \,  (z_d^k - C y_N^{*,k}, C \varphi)_{D}  \hspace{-10ex} \notag \\
	& && \forall \varphi \in Y_N, \ \forall k \in \xK, \label{OS_rb_strong:3} \\
	(u_N^{*} - u_d,\psi)_U  - m(\psi, p_N^{*,1}) & = 0 &&
	\forall \psi \in U_N^0, \label{OS_rb_strong:4}
\end{align}
\end{subequations}
where the final condition of the adjoint is given by $p_N^{*,K+1} = 0$. The reduced-order optimality system can be solved efficiently using an offline-online computational procedure which is briefly discussed in Section~\ref{ssec:comp_strong}.

Note that we use a single reduced basis ansatz and test space for the state and adjoint equations for two reasons: first, a single space for state and adjoint guarantees the stability of the reduced-order optimality system~\cite{GV2011a}; and second, the reduced-order optimality system~\eqref{OS_rb_strong} reflects the reduced-order  4D-Var problem~\eqref{rb_strong} only if the spaces of the state and adjoint equations are identical. Since the state and adjoint solutions need to be well-approximated using the single space $Y_N$, we combine both snapshots of the state and adjoint equations into the reduced basis space $Y_N$. 

\cbb{We also note that the dynamics of the state and adjoint are often different, and thus separate spaces for the state and adjoint would be beneficial concerning the computational efficiency, i.e. the dimension of the state/adjoint reduced basis space  and thus the overall dimension of the reduced-order optimality system would be considerably smaller. However, this requires a Petrov-Galerkin projection for the state and adjoint with associated detriment concerning the stability.}

\subsection{A~posteriori error estimation}
\label{sec:ee_strong}

We turn to the \textit{a~posteriori} error estimation procedure. Although we consider a parametrized problem here, we note that the error bounds proposed below can also be used in the non-parametrized reduced-order setting and are independent of how the reduced-order spaces are constructed, i.e., the bound directly applies to reduced-order approaches where the spaces are constructed e.g. using empirical orthogonal functions, POD, or dual-weighted POD~\cite{DN2008}. 

As mentioned above, our main goal is to rigorously bound the error in the optimal control, $u^* - u_N^*$. This will allow us to confirm the fidelity of the reduced-order 4D-Var solution efficiently during the online stage. Our \textit{a~posteriori} error bounds are also crucial in the construction of the reduced basis spaces by the POD-Greedy algorithm (see Section~\ref{ssec:greedy_strong}).

To begin, we require the residuals
\begin{align}
	\ryk(\phi;\mu) &= f(\phi) - a(\yNoptk,\phi;\mu) - \frac{1}{\tau}
	m(\yNoptk - \yNoptkm,\phi) \quad  \forall \phi \in Y, \ k \in \xK, \\
	\rpk(\varphi;\mu) &=  (z_d^k - C \yNoptk, C \varphi)_D
	- a(\varphi,\pNoptk;\mu) - \frac{1}{\tau} m(\varphi,\pNoptk-\pNoptkp)
	\notag \\ & \hspace{2.88in} \forall \varphi \in Y, \ k \in \xK, \\
	r_u(\psi;\mu) &= m(\psi,p_N^{*,1}) - (u_N^* - u_d, \psi)_\cU
	\quad  \forall \psi \in U.
\end{align}
We also define
\begin{equation}
	R_y = \Big( \tau \sum_{k=1}^K \norm{\ryk}_{Y'}^2 \Big)^{1/2},
	\qquad 
	R_p = \Big( \tau \sum_{k=1}^K \norm{\rpk}_{Y'}^2 \Big)^{1/2},
	\label{eq:bigR}
\end{equation}
and the errors $\eyk = y^{*,k} - \yNoptk$, $\epk = p^{*,k} - \pNoptk$, and $e_u = u^* - u_N^*$. Note that we use $\norm{r_{y,p}^k}_{Y'}$ and $\norm{r_u}_{U'}$ as a shorthand notation for $\norm{r_{y,p}^k(\cdot;\mu)}_{Y'}$ and $\norm{r_u(\cdot;\mu)})_{U'}$, respectively. We can now state our main result:

\begin{proposition}\label{prop:P1}
Let $u^*$ and $u_N^*$ be the optimal solutions of the full-order and reduced-order 4D-Var problems,~\refeq{fe_strong} and \refeq{rb_strong}, respectively. The error satisfies
\begin{equation} 
\norm{u^* - u_N^*}_U \leq  \Delta_N^u(\mu) := c_1(\mu) + \sqrt{c_1(\mu)^2 + c_2(\mu)} \quad \forall \mu \in \cD,
 \label{eq:eb_strong}
\end{equation}
where $c_1(\mu)$ and $c_2(\mu)$ are given by
\begin{align}
c_1(\mu) & = \frac{1}{2} \left( \norm{r_u(\cdot;\mu)}_{\cU'}
	+ \frac{1}{\sqrt{\alphaaLB}} R_p \right),  \quad \text{and} \label{eq:c1} \\ \quad c_2(\mu) & =  \left( \frac{\sqrt{2} + 1}{\alphaaLB} R_y R_p
	+ \frac{\gammac^2}{2 (\alphaaLB)^2} R_y^2 \right).
\label{eq:c2}
\end{align}
\end{proposition}

\begin{proof}
We start from the error-residual equations obtained from~\eqref{OS_fe_strong} and the definitions of the residuals
\begin{align}
	m(\eyk - \eykm,\phi) + \tau \, a(\eyk,\phi;\mu) &= \tau \, \ryk(\phi;\mu),
	 \qquad \qquad \forall \phi \in Y, \; k \in \xK, \label{err_res_a} \\
	m(\varphi,\epk - \epkp) + \tau \, a(\varphi,\epk;\mu) &= \tau \, \rpk(\varphi;\mu)
	-\tau \, (C \eyk, C \varphi)_D, \notag \\ & \hspace{1.35in} \forall \varphi \in Y,
	\; k \in \xK, \label{err_res_b} \\
	(e_u,\psi)_\cU - m(\psi,e_p^1) &= r_u(\psi;\mu), \hspace{0.62in} \forall \psi \in U,
	\label{err_res_c}
\end{align}
where $e_p^{K+1} = 0$ and $e_y^0 = e_u$. We first choose $\phi = \epk$ in \eqref{err_res_a} and take the sum from $k=1$ to $K$ to get
\begin{equation}
	\sum_{k=1}^K m(\eyk-\eykm,\epk) + \tau \sum_{k=1}^K a(\eyk,\epk;\mu)
	= \tau \sum_{k=1}^K \ryk(\epk;\mu). \label{eq1}
\end{equation}
Similarly, choosing $\varphi = \eyk$ in \eqref{err_res_b} and summing from $k=1$ to $K$  we obtain
\begin{equation}
	\sum_{k=1}^K m(\eyk,\epk-\epkp) + \tau \sum_{k=1}^K a(\eyk,\epk;\mu)
	= \tau \sum_{k=1}^K \rpk(\eyk;\mu)
	- \tau  \sum_{k=1}^K \norm{C \eyk}_D^2. \label{eq2}
\end{equation}
Finally, from~\eqref{err_res_c} with $\psi = e_u$ we have
\begin{equation}
	\norm{e_u}_\cU^2 - m(e_u,e_p^1) = r_u(e_u;\mu). \label{eq3}
\end{equation}
By adding equations \eqref{eq2} and \eqref{eq3}, and then subtracting \eqref{eq1} we get
\begin{multline}
	\sum_{k=1}^K m(\eykm,\epk) - \sum_{k=1}^K m(\eyk,\epkp) - m(e_u,e_p^1)
	+ \norm{e_u}_\cU^2 \label{eq4} \\
	= -\tau \sum_{k=1}^K \ryk(\epk;\mu) + \tau \sum_{k=1}^K \rpk(\eyk;\mu)
	+r_u(e_u;\mu)  - \tau  \sum_{k=1}^K \norm{C \eyk}_D^2.
\end{multline}
Since $e_y^0=e_u$, and $e_p^{K+1}=0$, the left-hand side of \eqref{eq4} reduces to $\norm{e_u}_\cU^2$ and we thus obtain
\begin{align}
	\norm{e_u}_\cU^2 + \tau  \sum_{k=1}^K \norm{C \eyk}_D^2
	&= -\tau \sum_{k=1}^K \ryk(\epk;\mu) + \tau \sum_{k=1}^K \rpk(\eyk;\mu)
	+r_u(e_u;\mu)  \notag \\
	&\le \Big( \tau \sum_{k=1}^K \norm{\ryk}_{Y'}^2 \Big)^{1/2}
	\Big( \tau \sum_{k=1}^K \norm{\epk}_Y^2 \Big)^{1/2} \notag \\
	&+ \Big( \tau \sum_{k=1}^K \norm{\rpk}_{Y'}^2 \Big)^{1/2}
	\Big( \tau \sum_{k=1}^K \norm{\eyk}_Y^2 \Big)^{1/2}
	+ \norm{r_u}_{\cU'} \norm{e_u}_\cU. \label{eq9}
\end{align}
From the proof for the spatio-temporal energy norm bound in~\cite{GP2005,KG2014} we know that
\begin{equation}
	\tau \sum_{k=1}^K \norm{\eyk}_Y^2
	\le \frac{\tau}{(\alphaaLB)^2} \sum_{k=1}^K \norm{\ryk}_{Y'}^2
	+ \frac{1}{\alphaaLB} \underbrace{m(e_u,e_u)}_{= \norm{e_u}_\cU^2}.
	\label{eq10}
\end{equation}
We need an analogous result for the adjoint. To this end, we first choose $\varphi = \epk$ in \eqref{err_res_b} to obtain
\begin{equation}
	m(\epk,\epk - \epkp) + \tau \, a(\epk,\epk;\mu) = \tau \, \rpk(\epk;\mu)
	- \tau \, (C \eyk, C \epk)_D. \label{eq5}
\end{equation}
We next note from the Cauchy-Schwarz inequality and Young's inequality that
\begin{equation}
	2 \, m(\epk,\epkp) \le m(\epk,\epk)+ m(\epkp,\epkp), \label{eq6}
\end{equation}
and also that
\begin{align}
	2 \, \tau \, (C \eyk, C \epk)_D
	& \le 2 \, \tau \,  \norm{C \eyk}_D \, \norm{C \epk}_D
	\le 2 \, \tau  \, \norm{C \eyk}_D \, \gammac \, \norm{\epk}_Y \notag
	\\ & \le \frac{2 \, \tau \, \gammac^2}{\alphaaLB} \norm{C \eyk}_D^2
	+ \frac{\tau \, \alphaaLB}{2} \norm{\epk}_Y^2, \label{eq7}
\end{align}
where we also used the definition of the constant $\gammac$. Finally, again from Young's inequality we obtain
\begin{equation}
	2 \, \tau \, \rpk(\epk;\mu) \le \frac{2 \, \tau}{\alphaaLB} \norm{\rpk}_{Y'}^2
	+ \frac{\tau \, \alphaaLB}{2} \norm{\epk}_Y^2. \label{eq8}
\end{equation}
By summing two times \eqref{eq5} from $k=1$ to $K$ and invoking \eqref{eq6}, \eqref{eq7}, and \eqref{eq8}, we obtain
\begin{equation}
	m(e_p^1,e_p^1) + \tau \sum_{k=1}^K a(\epk,\epk;\mu)
	\le \frac{2 \, \tau}{\alphaaLB} \sum_{k=1}^K \norm{\rpk}_{Y'}^2
	+ \frac{2 \, \tau \, \gammac^2}{\alphaaLB} \sum_{k=1}^K \norm{C \eyk}_D^2,
\end{equation}
and hence
\begin{equation}
	\tau \sum_{k=1}^K \norm{\epk}_Y^2
	\le \frac{2 \, \tau}{(\alphaaLB)^2} \sum_{k=1}^K \norm{\rpk}_{Y'}^2
	+ 2 \left( \frac{ \gammac}{\alphaaLB} \right)^2
	\tau \sum_{k=1}^K \norm{C \eyk}_D^2. \label{eq11}
\end{equation}
Using the inequalities \eqref{eq10} and \eqref{eq11} in \eqref{eq9}, invoking the definitions~\eqref{eq:bigR}, and noting that $(a^2 + b^2)^{1/2} \le \abs{a} + \abs{b}$, it follows that
\begin{align}
	\norm{e_u}_\cU^2 + \tau \sum_{k=1}^K \norm{C \eyk}_D^2
	\le& \norm{r_u}_{\cU'} \norm{e_u}_\cU + R_p \left[ \frac{1}{(\alphaaLB)^2} R_y^2
	+ \frac{1}{\alphaaLB} \norm{e_u}_\cU^2  \right]^{1/2} \notag \\ & + R_y \left[ \frac{2}{(\alphaaLB)^2} R_p^2
	+ 2 \left( \frac{ \gammac}{\alphaaLB} \right)^2
	\tau \sum_{k=1}^K \norm{C \eyk}_D^2 \right]^{1/2} \notag \\[1ex]
	\le& \norm{r_u}_{\cU'} \norm{e_u}_\cU + R_p \left[ \frac{1}{\alphaaLB} R_y
	+ \frac{1}{\sqrt{\alphaaLB}} \norm{e_u}_\cU \right] \notag \\ & + R_y \left[ \frac{\sqrt{2}}{\alphaaLB} R_p
	+ \frac{\sqrt{2} \, \gammac }{\alphaaLB}
	\Big( \tau \sum_{k=1}^K \norm{C \eyk}_D^2 \Big)^{1/2} \right].
	\label{eq:longineq}
\end{align}
We now use Young's inequality to bound 
\begin{equation}
	R_y \frac{\sqrt{2} \, \gammac}{\alphaaLB}
	\Big(  \tau \sum_{k=1}^K \norm{C \eyk}_D^2 \Big)^{1/2}
	\le \frac{\gammac^2}{2 (\alphaaLB)^2} R_y^2
	+  \tau \sum_{k=1}^K \norm{C \eyk}_D^2,
\end{equation}
and thereby eliminate the second term on the left-hand side of the inequality (\ref{eq:longineq}) to obtain
\begin{multline}
	\norm{e_u}_\cU^2 \le \norm{r_u}_{\cU'} \norm{e_u}_\cU
	+ \frac{1}{\sqrt{\alphaaLB}} R_p \norm{e_u}_\cU
	\\ + \frac{\sqrt{2} + 1}{\alphaaLB} R_y R_p
	+ \frac{\gammac^2}{2 (\alphaaLB)^2} R_y^2.
	\label{eq:shorterineq}
\end{multline}
Using the definitions of $c_1(\mu)$ and $c_2(\mu)$ in~\eqref{eq:c1} and~\eqref{eq:c2}, respectively, (\ref{eq:shorterineq}) simplifies to
\begin{align}
	\norm{e_u}_\cU^2 -  2 \, c_1(\mu) \, \norm{e_u}_\cU - c_2(\mu) \le 0.
\end{align}
We obtain the desired result by bounding the error $\norm{e_u}_{\cU}$ by the larger root of the quadratic inequality.
\end{proof}

\cbb{We note that we currently cannot assess the tightness of the error bound~\refeq{eq:eb_strong} by providing an {\it a priori} upper bound for the effectivity, i.e., the ratio of the bound to the error. We present numerical results for the effectivity in Section~\ref{ssec:rom4dvar}.}

\subsection{Computational Procedure}
\label{ssec:comp_strong}

We briefly comment on the computational procedure to solve the reduced-order 4D-Var problem and to evaluate the error bound. Given the affine parameter dependence, the offline-online decomposition for the reduced basis approximation is already quite standard in the reduced basis literature~\cite{RHP2008}; for the parabolic case considered in this paper, we also specifically refer to~\cite{GP2005,KG2014}. The evaluation of the \textit{a~posteriori} error bounds requires the following ingredients:
\begin{itemize}
	\item the dual norm of the residuals $\norm{\ryk}_{Y'}$, $\norm{\rpk}_{Y'}$,
	and $\norm{\ru}_{\cU'}$;
	\item the coercivity lower bound $\alphaaLB$ and the constant $\gammac$.
\end{itemize}
For the construction of the coercivity lower bound, $\alphaaLB$, various recipes exist~\cite{HRS+2007,PRV+2002,VRP2002}. The specific choices for our numerical tests are stated in Section~\ref{sec:results}. The constant $\gammac$ is parameter-independent and can be computed by solving a generalized eigenproblem. The offline-online evaluation of the dual norms of the residuals is standard and hence omitted~\cite{RHP2008}. For a summary of the computational cost in the parabolic optimal control context, we refer to~\cite{KG2014}. 

We solve the full-order and reduced-order 4D-Var problems with a preconditioned Newton-CG method on the ``reduced'' cost functional $j(u;\mu) := J(y(u),u;\mu)$, i.e., we eliminate the PDE-constraint in the minimization problem. The control mass matrix is used as a preconditioner. We present results for the number of CG iterations in Section~\ref{sec:results}. Overall, the online computational cost to solve the \textit{reduced-order} 4D-Var problem \textit{and} to evaluate the \textit{a~posteriori} error bound depends only on the reduced basis dimensions $N_Y$ and $N_U^0$, but is \textit{independent} of $\cN$.

\subsection{Greedy Algorithm}
\label{ssec:greedy_strong}

To construct the reduced basis spaces $Y_N$ and $U_N^0$, we use the POD-Greedy sampling procedure in Algorithm~\ref{alg:greedy_strong}. Here, $\Xi_\train \subset \cD$ is a finite but suitably large training sample, $\mu^1 \in \Xi_\train$ is the initial parameter value, $N_{\rm max}$ the maximum number of greedy iterations, and $\epsilon_\tolmin > 0$ a prescribed error tolerance. We also define the relative error bound $\Delta_{N,{\rm rel}}^u(\mu) = \Delta_N^u(\mu)/\norm{u_N^*(\mu)}_U$. Furthermore, for a given time history $v_k \in Y, \ k \in \xK$, the operator $\textrm{POD}_Y(\{ v_k: k \in \xK \})$ returns the largest POD-mode with respect to the $(\cdot,\cdot)_Y$ inner product (normalized with respect to the $Y$-norm), and $v^k_{{\textrm{proj}},N}(\mu)$ denotes the $Y$-orthogonal projection of~$v^k(\mu)$ onto the reduced basis space $Y_N$.

In steps~6 and 7 of Algorithm~\ref{alg:greedy_strong} we expand the reduced basis space $Y_N$ with the largest POD mode of both the state and the adjoint solution. Note that we apply the POD in these two steps to the time history of the optimal state and adjoint projection errors, i.e., $e^{y,k}_{\textrm{proj},N}(\mu) = y^{*,k}(\mu) - y^{*,k}_{\textrm{proj},N}(\mu)$ and $e^{p,k}_{\textrm{proj},N}(\mu) = p^{*,k}(\mu) - p^{*,k}_{\textrm{proj},N}(\mu), \ k \in \xK$, and not to the solutions $y^k(\mu),\ k \in \xK$, and  $p^k(\mu),\ k \in \xK$, itself.\footnote{For the first iteration of the algorithm we define $v^k_{{\textrm{proj}},0}(\mu) = 0$, and hence $e^{y,k}_{\textrm{proj},0}(\mu) = y^k(\mu)$ and $e^{p,k}_{\textrm{proj},0}(\mu) = p^k(\mu)$.} This ensures that the POD modes are already orthogonal with respect to the $(\cdot,\cdot)_Y$ inner product and that we add only new information to $Y_N$ which is not yet captured in the reduced basis.

In step~8 we expand the reduced basis space $U_N^0$ with the optimal control at $\mu^*$. Due to the time-dependence of the state and adjoint, it is possible that a specific parameter $\tilde{\mu}$ is picked several times by the greedy search in step~9. Before expanding $U_N^0$, we thus need to check if the new snapshot is already contained in the reduced basis space $U_{N-1}^0$, and consequently discard linearly dependent snapshots. By construction, we thus have $\dim(U_N^0) \le N$ and $\dim(Y_N) = 2N$ (although it is theoretically possible that $\dim(Y_N) \le 2N$, we did not observe this case in the numerical results). Finally, we note that information from the data assimilation cost functional enters through the adjoint equation and the adjoint snapshots into $Y_N$. 

\begin{algorithm}
\begin{onehalfspace}
\caption{Sampling Procedure: Strong-constraint 4D-Var}
\label{alg:greedy_strong}
\begin{algorithmic}[1]
\STATE Choose $\Xi_\train \subset \cD$, $\mu^1 \in \Xi_\train$, $N_{\rm max}$, and
$\epsilon_\tolmin > 0$
\STATE Set $N \leftarrow 0$, \quad $Y_N \leftarrow \{ \}$, \quad $U_N^0 \leftarrow \{ \}$
\STATE Set $\mu^* \leftarrow \mu^1$ \; and \; $\Delta_{N,{\rm rel}}^u(\mu^*) \leftarrow \infty$
\WHILE {$\Delta_{N,{\rm rel}}^u(\mu^*) > \epsilon_\tolmin$ {\bf and} $N \leq N_{\rm max}$}
\STATE $N \leftarrow N + 1$
\STATE $\zeta_1
= \textrm{POD}_Y\big(\big\{ e^{y,k}_{\textrm{proj},N-1}(\mu^*): k \in \xK
\big\}\big)$, \quad $Y_N \leftarrow Y_{N-1} \oplus \spn \{ \zeta_1 \}$
\STATE $\zeta_2
= \textrm{POD}_Y\big(\big\{ e^{p,k}_{\textrm{proj},N}(\mu^*): k \in \xK
\big\}\big)$, \quad \quad $Y_N \leftarrow Y_{N} \oplus \spn \{ \zeta_2 \}$
\STATE $U_N^0 \leftarrow U_{N-1}^0 \oplus \spn \{ u^*(\mu^*) \}$
\STATE $\displaystyle \mu^* \leftarrow
\operatorname*{arg \, max}_{\mu \in \Xi_\train}\; \Delta_{N,{\rm rel}}^u(\mu)$
\ENDWHILE
\end{algorithmic}
\end{onehalfspace}
\end{algorithm}

\section{Weak-constraint 4D-Var}
\label{sec:weak}

We next consider the weak-constraint 4D-Var data assimilation problem, thus accounting for possible model errors in the dynamical system. For simplicity, we assume in this section that the initial condition is known and that we are only interested in bounding the model error. We consider the combined problem (unknown initial condition \textit{and} model error) in the next section.

\subsection{Problem statement}
\label{ssec:ps_weak}

To emphasize the relation between the weak-constraint 4D-Var problem and the optimal control setting, we denote in this section the model error by $u$. However, the model error is now time-dependent, i.e., $u = u^k, \, k \in \xK$, and appears in every time step of the dynamical system. For a given parameter $\mu \in \cD$, the weak-constraint 4D-Var problem is then given by the minimization problem
\begin{gather}
\begin{aligned}
	\label{fe_weak}
	& \min_{y \in Y^K, \, u \in U^K} J(y,u;\mu)  \quad \text{s.t.} \quad y \in Y^K \quad\text{solves} \\ & m(y^k,v) + \tau \,
	a(y^k,v;\mu) = m(y^{k-1},v) + \tau \, b(u^k,v) + \tau \, f(v)  \\ & \hspace{3in} \forall v \in Y,
	\ \forall k \in \xK, 
\end{aligned}
\end{gather}
with initial condition $m(y^0,v) = m(y_0,v)$ for all $v \in Y,$ and cost functional $J(\cdot,\cdot;\mu): Y^K \times U^K \to \xR$ given by 
\begin{equation}
\label{cost_fe_weak}
 J(y,u;\mu) = \frac{\tau}{2} \sum_{k=1}^K \norm{u^k - u_d^k}_{U}^2	+ \frac{\tau}{2} \sum_{k=1}^K \norm{C y^k - z_d^k}^2_{D}.
\end{equation}
We note that the cost functional now contains the contribution of the model error $u^k$ as a sum over all time steps. In the optimal control setting, $u_d^k \in U, \, k \in \xK$ denotes the desired optimal control. In the data assimilation setting, however,  $u_d^k$ is usually set to zero since the model error is generally assumed to be unbiased~\cite{LSZ2015}. We also note that a constant (known) bias can be taken into account by adjusting the right-hand side $f(v)$. Similar to the strong-constraint formulation, $z_d^k \in D$, $k \in \xK$, are the observed outputs. 

We again obtain the associated necessary and sufficient first-order optimality conditions using a Lagrangian approach: Given $\mu \in \cD$, the optimal solution $(y^*,p^*,u^*) \in Y^K \times Y^K \times U^K$ satisfies
\begin{subequations}
\label{OS_fe_weak}
\begin{align}
	m(y^{*,k} - y^{*,k-1},\phi) + \tau \,  a ( y^{*,k},\phi;\mu)  & = \tau \, b(u^k,\phi) + \tau \, f(\phi) \hspace{-10ex} \notag \\
	& && \forall \phi \in Y, \ \forall k \in \xK, \label{OS_fe_weak:1} \\
	m(y^0,\phi) & = m(y_0,\phi) && \forall \phi \in Y,  \label{OS_fe_weak:2} \\
	m(\varphi, p^{*,k} - p^{*,k+1}) + \tau \,  a (\varphi,p^{*,k};\mu) \notag  &=
	\tau \,  (z_d^k - C y^{*,k}, C \varphi)_{D} \hspace{-10ex} \notag \\
	& && \forall \varphi \in Y, \ \forall k \in \xK, \label{OS_fe_weak:3} \\
	\tau \, (u^{*,k} - u_d^k,\psi)_U  - \tau \,  b(\psi, p^{*,k}) & = 0 && 
	\forall \psi \in U, \ \forall k \in \xK, \label{OS_fe_weak:4}
\end{align}
\end{subequations}
where the final condition of the adjoint is given by $p^{*,K+1} = 0$. We note that the adjoint equation of the weak-constraint formulation~\eqref{OS_fe_weak:3} is identical to the adjoint of the strong constraint formulation~\eqref{OS_fe_strong:3}.

\subsection{Reduced basis approximation}
\label{ssec:rb_weak}

We again assume that we are given the reduced basis spaces $Y_N \subset Y$ for the state and adjoint and $U_N \subset U$ for the control. Whereas the construction of the space $Y_N$ directly follows from the discussion in Section~\ref{ssec:greedy_strong} for the strong-constraint case, the construction of $U_N$ needs to be adjusted to account for the time-dependence of the model error.  We briefly outline the procedure in Section~\ref{ssec:greedy_weak}.

For a given parameter $\mu \in \cD$, we can now state the weak-constraint reduced-order 4D-Var data assimilation problem as follows
\begin{gather}
\begin{aligned}
	\label{rb_weak}
	& \min_{y_N \in Y_N^K, \, u_N \in U_N^K} J(y_N,u_N;\mu) \quad \text{s.t.} \quad y_N \in Y_N^K \quad\text{solves} \\ & m(y_N^k,v)
	+ \tau \, a(y_N^k,v;\mu) = m(y_N^{k-1},v) + \tau \, b(u_N^k,v) + \tau \, f(v)  \\ & \hspace{3in} \forall v \in
	Y_N, \ \forall k \in \xK,
\end{aligned}
\end{gather}
with initial condition $m(y_N^0,v) = m(y_0,v)$ for all $v \in Y_N$. The reduced-order optimality system directly follows from~\eqref{OS_fe_weak} and is thus omitted. 


\subsection{A~posteriori error estimation}
\label{ssec:ee_weak}

We first introduce the residuals for the weak-constraint case
\begin{align}
	\rykt(\phi;\mu) &= f(\phi) + b(\uNoptk,\phi) - a(\yNoptk,\phi;\mu) - \frac{1}{\tau}
	m(\yNoptk - \yNoptkm,\phi) \notag \\ & \hspace{2.88in} \forall \phi \in Y, \ k \in \xK, \label{res_y_weak} \\
	\rpkt(\varphi;\mu) &=  (z_d^k - C \yNoptk, C \varphi)_D
	- a(\varphi,\pNoptk;\mu) - \frac{1}{\tau} m(\varphi,\pNoptk-\pNoptkp)
	\notag \\ & \hspace{2.88in} \forall \varphi \in Y, \ k \in \xK,  \label{res_p_weak} \\
	\rukt(\psi;\mu) &= m(\psi,\pNoptk) - (\uNoptk - u_d, \psi)_\cU
	\quad  \forall \psi \in U, \ k \in \xK. \label{res_u_weak}
\end{align}
Since the adjoint equations~\eqref{OS_fe_weak:3} and~\eqref{OS_fe_strong:3} are identical, the adjoint residual is actually equivalent to the strong-constraint case, i.e., $\rpk = \rpkt$. Similar to~\eqref{eq:bigR}, we introduce the sums from $k =1$ to $K$ of the dual norms of the residuals as
\begin{equation}
\tilde{R}_{y,p} = \Big( \tau \sum_{k=1}^K \norm{\tilde{r}_{y,p}^k(\cdot;\mu)}_{Y'}^2 \Big)^{1/2}, \qquad 
	\tilde{R}_u = \Big( \tau \sum_{k=1}^K \norm{\rukt(\cdot;\mu)}_{U'}^2 \Big)^{1/2},
	\label{eq:bigR_w}
\end{equation}
and the time-dependent model error $\euk = u^{*,k} - \uNoptk$. We may now state our main result:

\begin{proposition}\label{prop:P2}
Let $u^{*,k}$ and $\uNoptk$, $k \in \xK$, be the optimal solutions of the full-order and reduced-order 4D-Var problems~\refeq{fe_weak} and \refeq{rb_weak}, respectively. The error satisfies
\begin{equation} 
\Big( \tau \sum_{k=1}^{K} \norm{u^{*,k} - \uNoptk}_U^2 \Big)^{1/2} \leq  \tilde{\Delta}_N^u(\mu) := c_1(\mu) + \sqrt{c_1(\mu)^2 + c_2(\mu)} \quad \forall \mu \in \cD,
 \label{eq:eb_weak}
\end{equation}
where $c_1(\mu)$ and $c_2(\mu)$ are given by
\begin{align}
c_1(\mu) & = \frac{1}{2} \left( \tilde{R}_u
	+ \frac{\sqrt{2} \, \gammab}{\alphaaLB} \tilde{R}_p \right),  \quad \text{and} \label{eq:c1_weak} \\ \quad c_2(\mu) & =  \left( \frac{2\sqrt{2}}{\alphaaLB} \tilde{R}_y \tilde{R}_p
	+ \frac{\gammac^2}{2 (\alphaaLB)^2} \tilde{R}_y^2 \right).
\label{eq:c2_weak}
\end{align}
\end{proposition}

\begin{proof}
The proof follows partly from the proof of Proposition~\ref{prop:P1}; we thus stress the differences and refer to the previous proof whenever possible. We again start from the error-residual equations which are now given by 
\begin{align}
	m(\eyk - \eykm,\phi) + \tau \, a(\eyk,\phi;\mu) &= \tau \, \rykt(\phi;\mu) + \tau \, b(\euk,\phi) \hspace{-10ex} \notag \\ ,
	 & && \forall \phi \in Y, \; k \in \xK, \label{err_res_a_w} \\
	m(\varphi,\epk - \epkp) + \tau \, a(\varphi,\epk;\mu) &= \tau \, \rpkt(\varphi;\mu)
	-\tau \, (C \eyk, C \varphi)_D, \hspace{-10ex} \notag \\ & &&  \forall \varphi \in Y,
	\; k \in \xK, \label{err_res_b_w} \\
	\tau \, (\euk,\psi)_\cU - \tau \, b(\psi,\epk) &= \tau \, \rukt(\psi;\mu), && \forall \psi \in U, \; k \in \xK,
	\label{err_res_c_w}
\end{align}
where $e_p^{K+1} = 0$ and $e_y^0 = 0$, since we guarantee that $y_0 \in Y_N$. We now choose $\phi = \epk$ in \eqref{err_res_a_w}, $\varphi = \epk$ in \eqref{err_res_b_w}, and 
$\psi = \euk$ in~\eqref{err_res_c}, sum all equations from from $k=1$ to $K$ and combine them following the proof of Proposition~\ref{prop:P1} to obtain
\begin{align}
	\tau \sum_{k=1}^{K} & \norm{\euk}_\cU^2 + \tau  \sum_{k=1}^K \norm{C \eyk}_D^2
	\notag \\ &= -\tau \sum_{k=1}^K \rykt(\epk;\mu) + \tau \sum_{k=1}^K \rpkt(\eyk;\mu)
	 + \tau \sum_{k=1}^K \rukt(\euk;\mu) \notag \\
	&\le \tilde{R}_y
	\Big( \tau \sum_{k=1}^K \norm{\epk}_Y^2 \Big)^{1/2} + \tilde{R}_p
	\Big( \tau \sum_{k=1}^K \norm{\eyk}_Y^2 \Big)^{1/2}
	+ \tilde{R}_u \Big( \tau \sum_{k=1}^K \norm{\euk}_U^2 \Big)^{1/2}. \label{eq9_w}
\end{align}
We next bound the primal error. Since the primal equation contains the model error on the right-hand side, we need to extend the proof from~\cite{GP2005} for the spatio-temporal energy norm bound to include the extra term on the right-hand side. The derivation is similar to the one for the bound of the adjoint in the proof of Proposition~\ref{prop:P1} (cf.\ \eqref{eq5}~--~\eqref{eq11}), but instead of bounding the $(\cdot,\cdot)_D$ inner product using Cauchy-Schwarz and the constant $\gammac$, we invoke the continuity of the bilinear form $b(\cdot,\cdot)$. We can thus derive the bound
\begin{equation}
	\tau \sum_{k=1}^K \norm{\eyk}_Y^2
	\le \frac{2 \, \tau}{(\alphaaLB)^2} \sum_{k=1}^K \norm{\ryk}_{Y'}^2
	+ 2 \left( \frac{ \gammab}{\alphaaLB} \right)^2
	\tau \sum_{k=1}^K \norm{\euk}_U^2. \label{eq11_w}
\end{equation}
Furthermore, since the adjoint of the strong- and weak-constraint case are equivalent, we can directly use the bound~\eqref{eq11}. Using the inequalities \eqref{eq11_w} and \eqref{eq11} in \eqref{eq9_w},  invoking the definitions~\eqref{eq:bigR_w}, and noting that $(a^2 + b^2)^{1/2} \le \abs{a} + \abs{b}$, it follows that
\begin{align}
	\tau \sum_{k=1}^{K} \norm{\euk}_\cU^2 + & \tau \sum_{k=1}^K \norm{C \eyk}_D^2 \notag \\
	& \le \left[ \tilde{R}_u + \frac{\sqrt{2} \, \gammab}{\alphaaLB} \tilde{R}_p \right] \Big( \tau \sum_{k=1}^K \norm{\euk}_U^2 \Big)^{1/2}  \notag 
	\\ & \quad +\frac{2 \sqrt{2}}{\alphaaLB} \tilde{R}_y \tilde{R}_p + \frac{\sqrt{2} \, \gammac }{\alphaaLB}	\tilde{R}_y \Big( \tau \sum_{k=1}^K \norm{C \eyk}_D^2 \Big)^{1/2}.
	\label{eq:longineq2}
\end{align}
We again use Young's inequality to bound 
\begin{equation}
	\tilde{R}_y \frac{\sqrt{2} \, \gammac}{\alphaaLB}
	\Big(  \tau \sum_{k=1}^K \norm{C \eyk}_D^2 \Big)^{1/2}
	\le \frac{\gammac^2}{2 (\alphaaLB)^2} \tilde{R}_y^2
	+  \tau \sum_{k=1}^K \norm{C \eyk}_D^2,
	\label{eq:shorterineq2}
\end{equation}
and thereby eliminate the second term on the left-hand side of (\ref{eq:longineq2}) to obtain
\begin{multline}
	\tau \sum_{k=1}^{K} \norm{\euk}_\cU^2 \le \left[
	\tilde{R}_u + \frac{\sqrt{2} \, \gammab}{\alphaaLB} \tilde{R}_p \right] \Big( \tau \sum_{k=1}^K \norm{\euk}_U^2 \Big)^{1/2}  
	\\ + \frac{2 \sqrt{2}}{\alphaaLB} \tilde{R}_y \tilde{R}_p
	+ \frac{\gammac^2}{2 (\alphaaLB)^2} \tilde{R}_y^2.
\end{multline}
Using the definitions of $c_1(\mu)$ and $c_2(\mu)$ in~\eqref{eq:c1_weak} and~\eqref{eq:c2_weak}, respectively, we obtain
\begin{align}
	\tau \sum_{k=1}^{K} \norm{\euk}_\cU^2 -  2 \, c_1(\mu) \, \Big( \tau \sum_{k=1}^K \norm{\euk}_U^2 \Big)^{1/2} - c_2(\mu) \le 0,
\end{align}
The desired result follows again by using the larger root of the quadratic inequality as a bound for the error.

\end{proof}

The offline-online computational procedure in the weak-constraint case is analogous to the strong-constraint case discussed in Section~\ref{ssec:comp_strong} and therefore omitted. Note that we additionally require the constant $\gammab$ now, which is parameter-independent and can be computed by solving a generalized eigenproblem (similar to $\gammac$). For the Newton-CG method, we use the block-diagonal matrix ${\rm blkdiag}(\tau  \rmM, \ldots, \tau \rmM)$ as a preconditioner. 

\cbb{Similar to the strong-constraint case, we again cannot assess the tightness of the error bound~\refeq{eq:eb_weak} by providing an {\it a priori} upper bound for the associated effectivity. Instead, we present numerical results for the weak-constraint case also in Section~\ref{ssec:rom4dvar}.}

\subsection{Greedy Algorithm}
\label{ssec:greedy_weak}

The POD-Greedy sampling procedure to construct the reduced basis spaces $Y_N$ and $U_N$ in the weak-constraint case is very similar to the strong-constraint case. We summarize the procedure in Algorithm~\ref{alg:greedy_weak} and only comment on the differences.

First, since we assume in this section that the initial condition $y_0$ is known, we initialize the reduced basis space $Y_N$ with $y_0/\|y_0\|_Y$. Second, we additionally require the operator $\textrm{POD}_U(\{ v_k: k \in \xK \})$, which returns the largest POD mode with respect to the $(\cdot,\cdot)_U$ inner product (and normalized with respect to the $U$-norm). Also, $v^k_{{\textrm{proj}_U},N}(\mu)$ denotes the $U$-orthogonal projection of~$v^k(\mu)$ onto the reduced basis space $U_N$ and  $e^{u,k}_{\textrm{proj}_U,N}(\mu) = u^{*,k}(\mu) - u^{*,k}_{\textrm{proj}_U,N}(\mu)$ denotes the time history of the optimal model-error forcing. Since the model-error forcing  is time-dependent, we simply replace step~8 in Algorithm~\ref{alg:greedy_strong} with a POD-step and add only the largest POD mode $\zeta$ to $U_N$. We note that the POD modes $\zeta$ are orthogonal with respect to the $(\cdot,\cdot)_U$ inner product and that we now usually have $\dim(\cU_N) = N$ and $\dim(Y_N) = 2N + 1$ (due to the initial condition), i.e., the reduced basis space $U_N$ is enriched in every greedy step. Again, it is theoretically possible that $\dim(\cU_N) \le N$ and $\dim(Y_N) \le 2N + 1$, although we did not observe this case in the numerical results.

\begin{algorithm}
\begin{onehalfspace}
\caption{Sampling Procedure: Weak-constraint 4D-Var}
\label{alg:greedy_weak}
\begin{algorithmic}[1]
\STATE Choose $\Xi_\train \subset \cD$, $\mu^1 \in \Xi_\train$, $N_{\rm max}$, and
$\epsilon_\tolmin > 0$
\STATE Set $N \leftarrow 0$, \quad $Y_N \leftarrow \{ y_0/\|y_0\|_Y \}$, \quad $U_N \leftarrow \{ \}$
\STATE Set $\mu^* \leftarrow \mu^1$ \; and \; $\tilde{\Delta}_{N,{\rm rel}}^u(\mu^*) \leftarrow \infty$
\WHILE {$\tilde{\Delta}_{N,{\rm rel}}^u(\mu^*) > \epsilon_\tolmin$ {\bf and} $N \leq N_{\rm max}$}
\STATE $N \leftarrow N + 1$
\STATE $\zeta_1
= \textrm{POD}_Y\big(\big\{ e^{y,k}_{\textrm{proj},N-1}(\mu^*): k \in \xK
\big\}\big)$, \quad \ \ $Y_N \leftarrow Y_{N-1} \oplus \spn \{ \zeta_1 \}$
\STATE $\zeta_2
= \textrm{POD}_Y\big(\big\{ e^{p,k}_{\textrm{proj},N}(\mu^*): k \in \xK
\big\}\big)$, \quad \quad \ \ $Y_N \leftarrow Y_{N} \oplus \spn \{ \zeta_2 \}$
\STATE $\zeta_{\phantom{1}} = \textrm{POD}_U\big(\big\{ e^{u,k}_{\textrm{proj}_U,N-1}(\mu^*): k \in \xK
\big\}\big)$, \quad  $U_N \leftarrow U_{N-1} \oplus \spn \{ \zeta \}$
\STATE $\displaystyle \mu^* \leftarrow
\operatorname*{arg \, max}_{\mu \in \Xi_\train}\; \tilde{\Delta}_{N,{\rm rel}}^u(\mu)$
\ENDWHILE
\end{algorithmic}
\end{onehalfspace}
\end{algorithm}

\section{Combined 4D-Var formulation}
\label{sec:comb}

We now combine the results from the previous two sections and consider the classical 4D-Var data assimilation problem including model error. 



\subsection{Problem statement}

For a given parameter $\mu \in \cD$, we now consider the minimization problem
\begin{gather}
\begin{aligned}
	\label{fe_comb}
	& \min_{y \in Y^K, \, u \in U^{K+1}} J(y,u;\mu)  \quad \text{s.t.} \quad y \in Y^K \quad\text{solves} \\ & m(y^k,v) + \tau \,
	a(y^k,v;\mu) = m(y^{k-1},v) + \tau \, b(u^k,v) + \tau \, f(v)  \\ & \hspace{3in} \forall v \in Y,
	\ \forall k \in \xK, 
\end{aligned}
\end{gather}
with initial condition $m(y^0,v) = m(u^0,v)$ for all $v \in Y,$ and cost functional $J(\cdot,\cdot;\mu): Y^K \times U^{K+1} \to \xR$ given by 
\begin{equation}
\label{cost_fe_comb}
 J(y,u;\mu) = \frac{1}{2} \norm{u^0 - u_d^0}_{U}^2 +  \frac{\tau}{2} \sum_{k=1}^K \norm{u^k - u_d^k}_{U}^2	+ \frac{\tau}{2} \sum_{k=1}^K \norm{C y^k - z_d^k}^2_{D}.
\end{equation}
In addition to the error between the predicted and observed outputs, the cost functional now contains the deviation of the initial condition from the background state, $u_d^0 \in U,$ as well as  the model error for all time steps. As mentioned earlier, in the data assimilation context we usually have $u_d^0 \neq 0$ and $u_d^k = 0, \, 1 \leq k \leq K$, i.e. the background state is nonzero whereas the model error is assumed to have zero mean. 

The associated necessary and sufficient first-order optimality conditions are thus: Given $\mu \in \cD$, the optimal solution $(y^*,p^*,u^*) \in Y^K \times Y^K \times U^{K+1}$ satisfies
\begin{subequations}
\label{OS_fe_comb}
\begin{align}
	m(y^{*,k} - y^{*,k-1},\phi) + \tau \,  a ( y^{*,k},\phi;\mu)  & = \tau \, b(u^k,\phi) + \tau \, f(\phi) \hspace{-10ex} \notag \\ 
	& && \forall \phi \in Y, \ \forall k \in \xK, \label{OS_fe_comb:1} \\ 
	m(y^{*,0},\phi) & = m(u^0,\phi) && \forall \phi \in Y \label{OS_fe_comb:2} \\ 
	m(\varphi, p^{*,k} - p^{*,k+1}) + \tau \,  a (\varphi,p^{*,k};\mu)  & =
	\tau \,  (z_d^k - C y^{*,k}, C \varphi)_{D} \hspace{-10ex} \notag \\
	& && \forall \varphi \in Y, \ \forall k \in \xK, \label{OS_fe_comb:3} \\
	\tau \, (u^{*,k} - u_d^k,\psi)_U  - \tau \,  b(\psi, p^{*,k}) & = 0 && 
	\forall \psi \in U, \ \forall k \in \xK, \label{OS_fe_comb:4}  \\
	(u^{*,0} - u_d^0,\psi)_U  - m(\psi, p^{*,1}) & = 0 && 
	\forall \psi \in U, \label{OS_fe_comb:5}
\end{align}
\end{subequations}
where the final condition of the adjoint is given by $p^{*,K+1} = 0$.

\subsection{Reduced basis approximation and error estimation}

The reduced-order problem follows directly from~\eqref{fe_comb} and~\eqref{cost_fe_comb} by restricting the state, adjoint, and control spaces to their respective reduced basis spaces. We again introduce an integrated space $Y_N$ for the state and adjoint, and two separate spaces for the ``control,'' i.e., $U_N^0$ for the initial condition $u_N^0$ and $U_N$ for the model error $u_N^k, \, k \in \xK$. The greedy procedure to generate these spaces simply combines the algorithms introduced in Sections~\ref{ssec:greedy_strong} and~\ref{ssec:greedy_weak}.

For any given $\mu \in \cD$, we can now state the reduced-order minimization problem as follows
\begin{gather}
\begin{aligned}
	\label{rb_comb}
	& \min_{y_N \in Y_N^K, \, u_N \in U_N^0 \times U_N^K} J(y_N,u_N;\mu) \quad \text{s.t.}
	\quad y_N \in Y_N^K \quad\text{solves} \\
	& m(y_N^k,v) + \tau \, a(y_N^k,v;\mu) = m(y_N^{k-1},v) + \tau \, b(u_N^k,v) + \tau \, f(v)  \\
	& \hspace{3in} \forall v \in Y_N, \ \forall k \in \xK,
\end{aligned}
\end{gather}
with initial condition $m(y_N^0,v) = m(u_N^0,v)$ for all $v \in Y_N$. The reduced-order optimality system directly follows from~\eqref{OS_fe_comb} and is thus omitted.

The \textit{a~posteriori} error bound result is a combination of the strong- and weak-constraint case.  In addition to the residuals of the state $\rykt$, adjoint $\rpkt$, and model error $\rukt$ defined in~\eqref{res_y_weak}, \eqref{res_p_weak}, and \eqref{res_u_weak}, we also require the residual 
\begin{equation}
r_u^0(\psi;\mu) = m(\psi,p_N^{*,1}) - (u_N^{*,0} - u_d^0, \psi)_\cU
	\quad  \forall \psi \in U.
\end{equation}
The \textit{a~posteriori} error bound is given in the following proposition.

\begin{proposition}\label{prop:P3}
Let $u^{*,k}$ and $\uNoptk$ be the optimal solutions of the full-order and reduced-order 4D-Var problems~\refeq{fe_comb} and \refeq{rb_comb}, respectively. The error satisfies
\begin{multline} 
\Big( \norm{u^{*,0} - u_N^{*,0}}_U^2 + \tau \, \sum_{k=1}^{K} \norm{u^{*,k} - \uNoptk}_U^2 \Big)^{1/2} \\ \leq  \hat{\Delta}_N^u(\mu) := c_1(\mu) + \sqrt{c_1(\mu)^2 + c_2(\mu)} \quad \forall \mu \in \cD,
 \label{eq:eb_comb}
\end{multline}
where $c_1(\mu)$ and $c_2(\mu)$ are given by
\begin{equation}
c_1(\mu) = \frac{1}{2} \left( \left( \norm{r_u^0(\cdot;\mu)}_{U'}^2 + \tilde{R}_u^2 \right)^{1/2}
	+ \left(\frac{2 \, \gammab^2}{(\alphaaLB)^2} + \frac{1}{\alphaaLB} \right)^{1/2} \tilde{R}_p \right) \label{eq:c1_comb} 
\end{equation}
and
\begin{equation} c_2(\mu) =  \left( \frac{2\sqrt{2}}{\alphaaLB} \tilde{R}_y \tilde{R}_p
	+ \frac{\gammac^2}{2 (\alphaaLB)^2} \tilde{R}_y^2 \right).
\label{eq:c2_comb}
\end{equation}
\end{proposition}

The proof follows from the proofs of Propositions~\ref{prop:P1} and~\ref{prop:P2} and is thus omitted. The offline-online decomposition is analogous to our previous discussion in Section~\ref{ssec:comp_strong}.

\section{Numerical results}
\label{sec:results}


\subsection{Problem description}

We consider the dispersion of a pollutant governed by a convection-diffusion equation with a Taylor-Green vortex velocity field. The concentration of the pollutant is measured at five spatial locations over time. The computational domain is $\Omega = (-1,1)^2$ and we assume homogeneous Dirichlet boundary conditions on the lower boundary $\Gamma_D$ and homogeneous Neumann boundary conditions on the remaining boundary $\Gamma_N$. The P\'{e}clet number serves as our parameter, i.e., we have $\mu = {\rm Pe} \in \cD = [10,50]$. The bilinear form $a$ is thus given by
\begin{equation}
	a(w,v;\mu) = \frac{1}{\mu} \int_\Omega \nabla w \cdot \nabla v \dx
	+ \int_\Omega (\beta \cdot \nabla w) v \dx,
\end{equation}
and the velocity field is $\beta(x) = ( \sin(\pi x_1) \cos(\pi x_2), - \cos(\pi x_1) \sin(\pi x_2) )^T$. The domain $\Omega$ with measurement sites as well as the velocity field are sketched in Figure~\ref{fig:domain}. \cbb{Our model problem is motivated by the source reconstruction of a (possibly) accidental release of an agent, where the velocity field is known~\cite{Krysta2006,Krysta2007}. Although we consider a fixed velocity field here, our problem formulation also directly applies to (affinely) parametrized velocity fields.}

We do not consider an additional forcing term and thus set $f \equiv 0$. The inner product on $Y_\rme = \{ v \in H^1(\Omega): v|_{\Gamma_D} \equiv 0 \}$ is defined as $(w,v)_Y = \frac{1}{2} a(w,v;\muref) + \frac{1}{2} a(v,w;\muref)$ for the reference parameter $\muref = 30$. Since $\beta$ is divergence-free and $\beta \cdot n \equiv 0$ on $\Gamma$, one can show that
$a$ is coercive and that the symmetric part of $a$ is given by $1 / \mu \int_\Omega \nabla w \cdot \nabla v \dx$. Hence we can use the min-theta approach to construct a coercivity lower bound: $\alphaaLB := \muref / \mu$. For details, we refer to Appendix B.3 of \cite{Kaercher2016}.

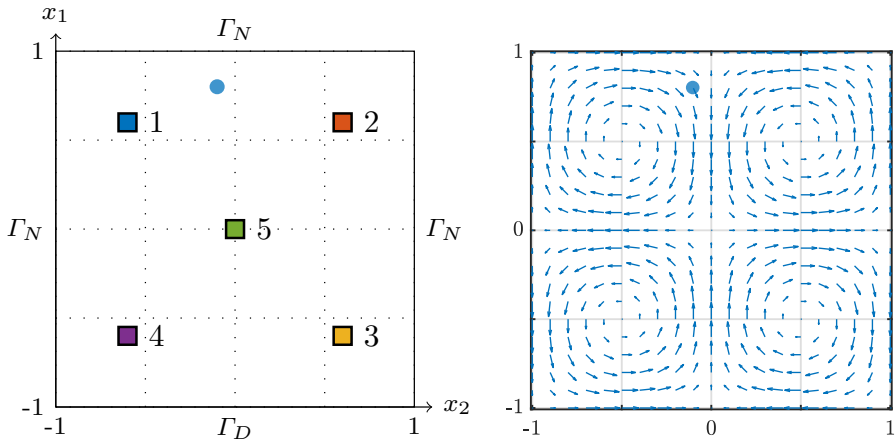
\begin{figure}[ht]
\resizebox{0.55\textwidth}{!}{%
\begin{tikzpicture}
	
	\definecolor{sensorcolor1}{RGB}{  0, 114, 189}
	\definecolor{sensorcolor2}{RGB}{217,  83,  25}
	\definecolor{sensorcolor3}{RGB}{237, 177,  32}
	\definecolor{sensorcolor4}{RGB}{126,  47, 142}
	\definecolor{sensorcolor5}{RGB}{119, 172,  48}
	
	\draw[step=1cm, loosely dotted] (-2,-2) grid (2,2);
	
	\draw [<->] (-2,2.2) node (yaxis) [above] {\small $x_1$}
	|- (2.2,-2) node (xaxis) [right] {\small $x_2$};
	\draw ( 2,-2) node [below] {\small 1};
	\draw (-2, 2) node [left]  {\small 1};
	\draw (-2,-2) node [below] {\small -1};
	\draw (-2,-2) node [left]  {\small -1};
	
	\draw ( 0,-2) node [below] {\small $\Gamma_D$};
	\draw (-2, 0) node [left]  {\small $\Gamma_N$};
	\draw ( 2, 0) node [right] {\small $\Gamma_N$};
	\draw ( 0, 2) node [above] {\small $\Gamma_N$};
	
	\draw (-2,-2) -- (2,-2) -- (2,2) -- (-2,2) -- (-2,-2);
	
	\draw (-0.2,1.6) node [inner sep=0pt,minimum size=1.4mm,shape=circle,
	draw=sensorcolor1!75,fill=sensorcolor1!75,thick,label=right:{}] {};
	
	\draw (-1.2,1.2) node [inner sep=0pt,minimum size=2mm,shape=rectangle,
	draw=black,fill=sensorcolor1,thick,label=right:{1}] {};
	
	\draw (1.2,1.2) node [inner sep=0pt,minimum size=2mm,shape=rectangle,
	draw=black,fill=sensorcolor2,thick,label=right:{2}] {};
	
	\draw (1.2,-1.2) node [inner sep=0pt,minimum size=2mm,shape=rectangle,
	draw=black,fill=sensorcolor3,thick,label=right:{3}] {};
	
	\draw (-1.2,-1.2) node [inner sep=0pt,minimum size=2mm,shape=rectangle,
	draw=black,fill=sensorcolor4,thick,label=right:{4}] {};
	
	\draw (0,0) node [inner sep=0pt,minimum size=2mm,shape=rectangle,
	draw=black,fill=sensorcolor5,thick,label=right:{5}] {};
\end{tikzpicture}
} 
\hfill
\raisebox{0.9mm}{\includegraphics[width=0.45\textwidth]{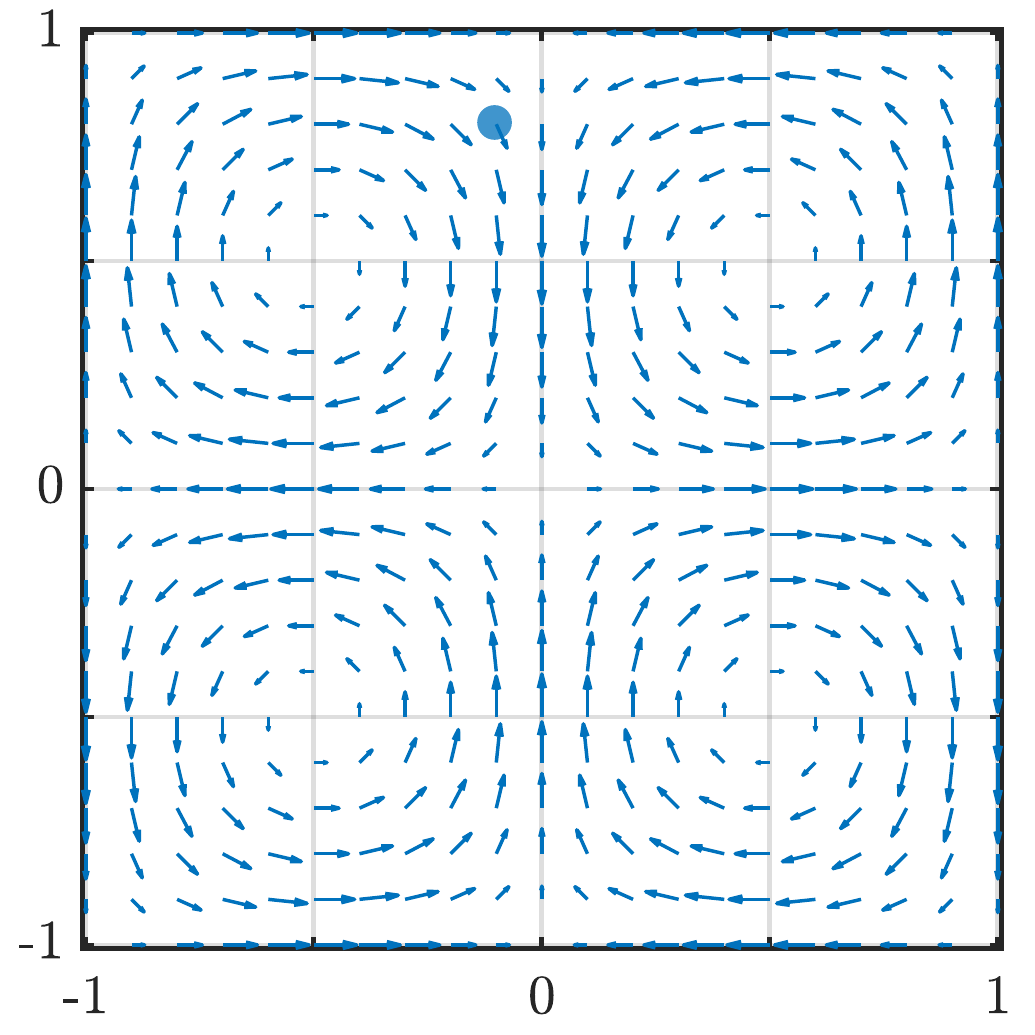}}
\caption{Left: Sketch of the computational domain with measurement locations $\Omega_1,\dots,\Omega_5$. The centers of the sensors are located at $(\pm 0.6,\pm 0.6)^T$ and $(0,0)^T$; their width and height is 0.1. The colors match those in Figure~\ref{fig:outputs}. Right: Plot of the Taylor-Green vortex velocity field. The blue dot indicates the center $(-0.1,0.8)^T$ of the Gaussian serving as initial condition.}
\label{fig:domain}
\end{figure}

We choose the time interval $I = [0,8]$ and a time step size $\tau = 0.04$ resulting in $K = 200$ time steps. For the space discretization we introduce a spatial mesh with an element size of $h = 0.04$ and corresponding linear finite element approximation spaces $Y=U$ with $\cN_Y = \cN_U = 13,131$ degrees of freedom. We assume that the (unknown true) initial condition $y_0^\textrm{true}$ is given by a spatial Gaussian function with mean $(-0.1,0.8)^T$ and covariance matrix $\sigma^2 \xI$, where $\sigma = 0.1$ and $\xI$ is the identity matrix (the center of the Gaussian is shown as a blue dot in Figure~\ref{fig:domain}). The average concentration over the measurement domains shown in Figure~\ref{fig:domain} serve as our five outputs $h_i(\phi) = \abs{\Omega_i}^{-1} \int_{\Omega_i} \phi \dx$, $i=1,\dots,5$. We then generate noisy measurements by adding white noise to the outputs computed from the full-order model for the (unknown true) parameter $\mu^\textrm{true} = 30$ with initial condition $y_0^\textrm{true}$ such that $z_d^k = C y^{k,\textrm{true}} + \eta^k$, where $\eta^k \in \xR^5, \ k \in \xK,$ is a vector containing uncorrelated Gaussian noise in each entry, i.e, $\eta_i^k \sim N(0,0.05^2), \ i = 1,\ldots,5, \ k \in \xK$ . The inverse observation covariance matrix is given by $\rmD = 10 \xI$. \cbb{In practice, the choice 10 produces acceptable results for the 4D Var problem (a thorough discussion of the impact of Tychonov regularization on 4D-Var is beyond the scope of this paper, we refer to~\cite{Puel2009} for more details).} In the strong-constraint case, we assume an optimal prior and set the prior mean $u_d$ to be equal to the true initial condition. In the weak-constraint case, we set $b(\cdot,\cdot) = m(\cdot,\cdot)$ to account for the model-error forcing and $u_d^k = 0, \ k \in \xK,$ i.e. the model-error forcing is assumed to be unbiased and have zero mean. In both cases, the inverse prior covariance matrix $\rmU$ is given by the mass matrix.

\begin{figure}[ht]
	\begin{tabular}{lccc}
		& $\mu=10$ & \hspace*{-1cm} $\mu^\textrm{true}=30$ & \hspace*{-1cm} $\mu=50$ \\
		\raisebox{1.4cm}{\rotatebox[origin=c]{90}{$k=20$}} &
		\hspace*{-0.5cm} \includegraphics[width=0.35\textwidth]{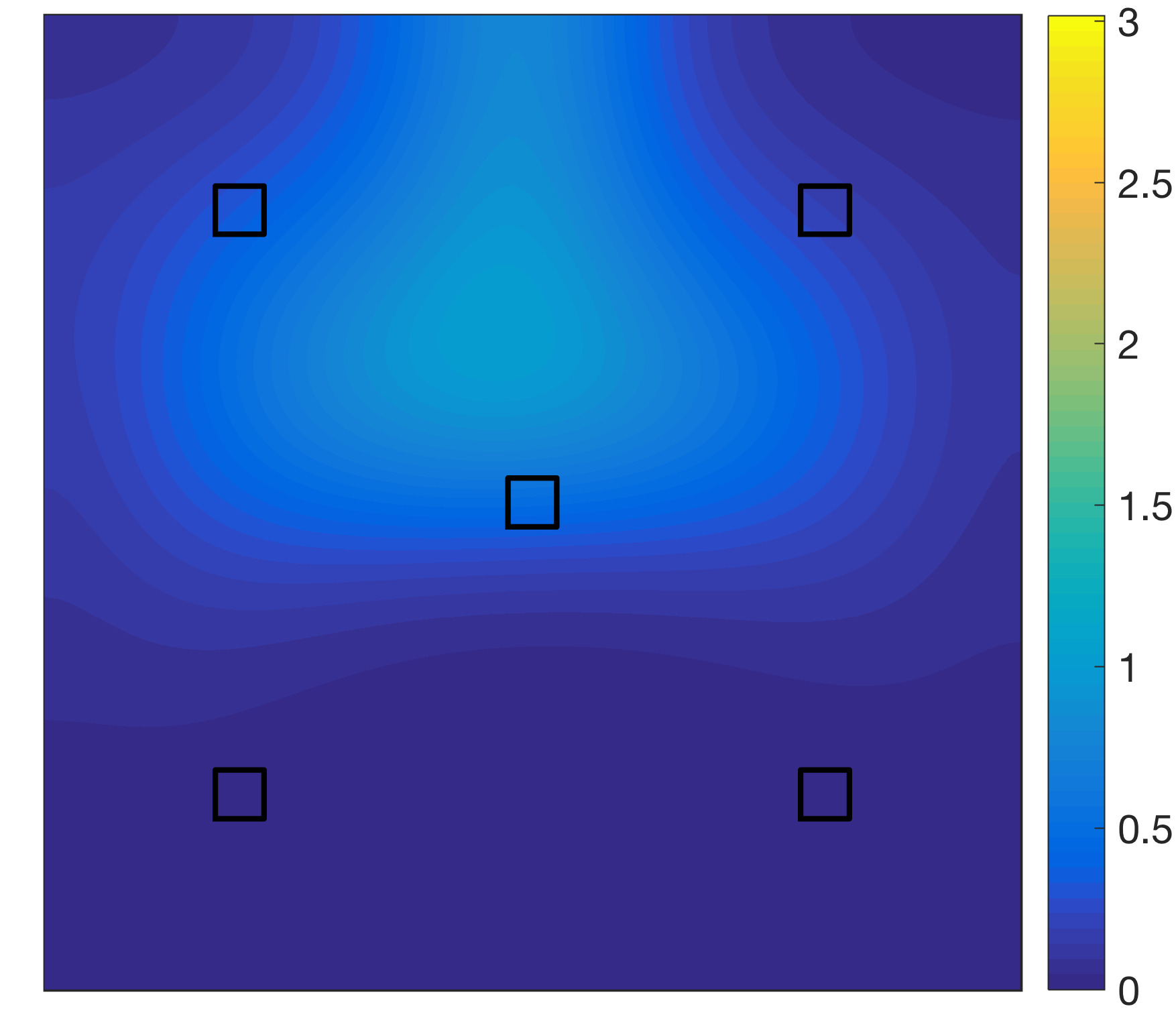} &
		\hspace*{-1cm} \includegraphics[width=0.35\textwidth]{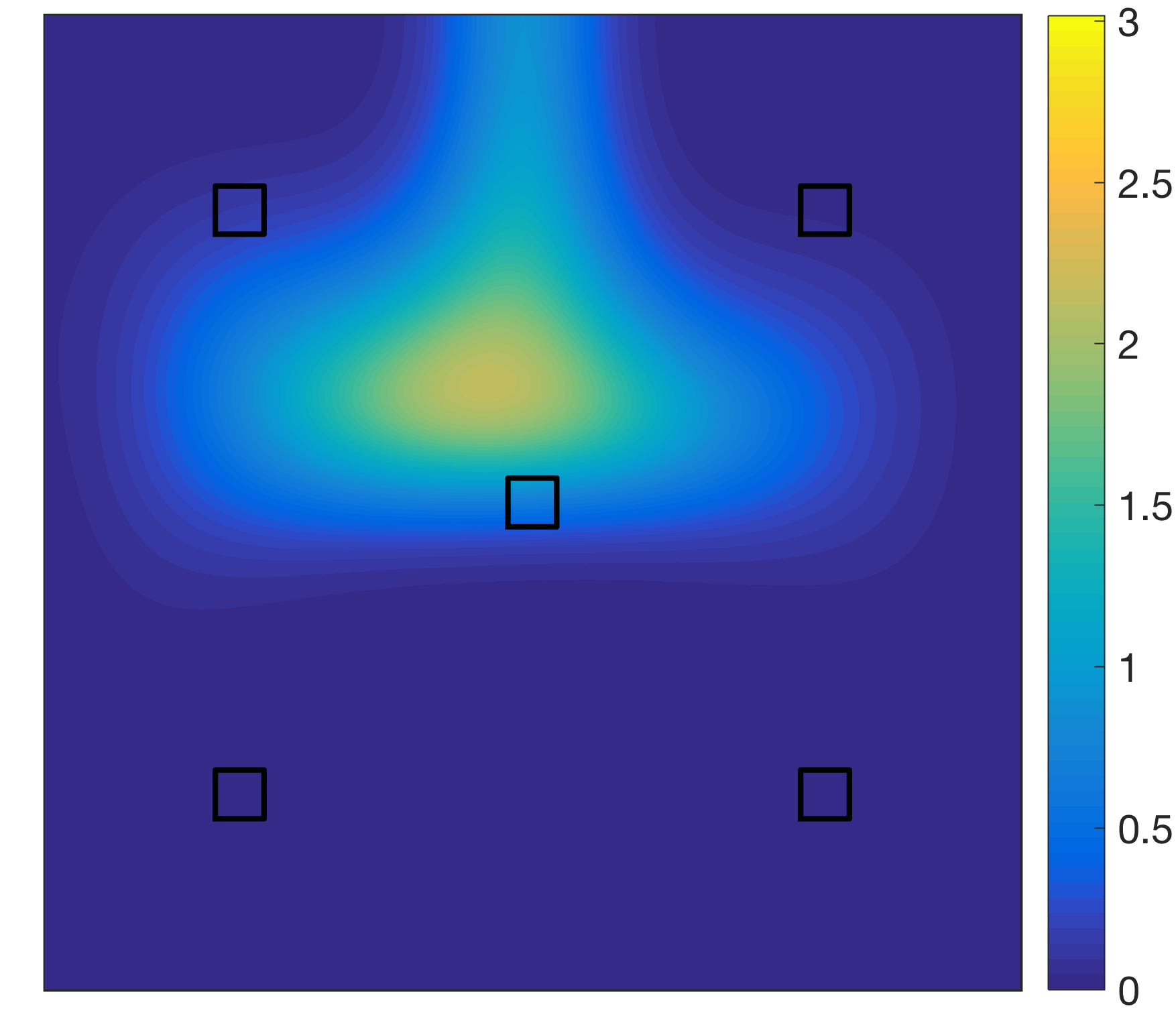} &
		\hspace*{-1cm} \includegraphics[width=0.35\textwidth]{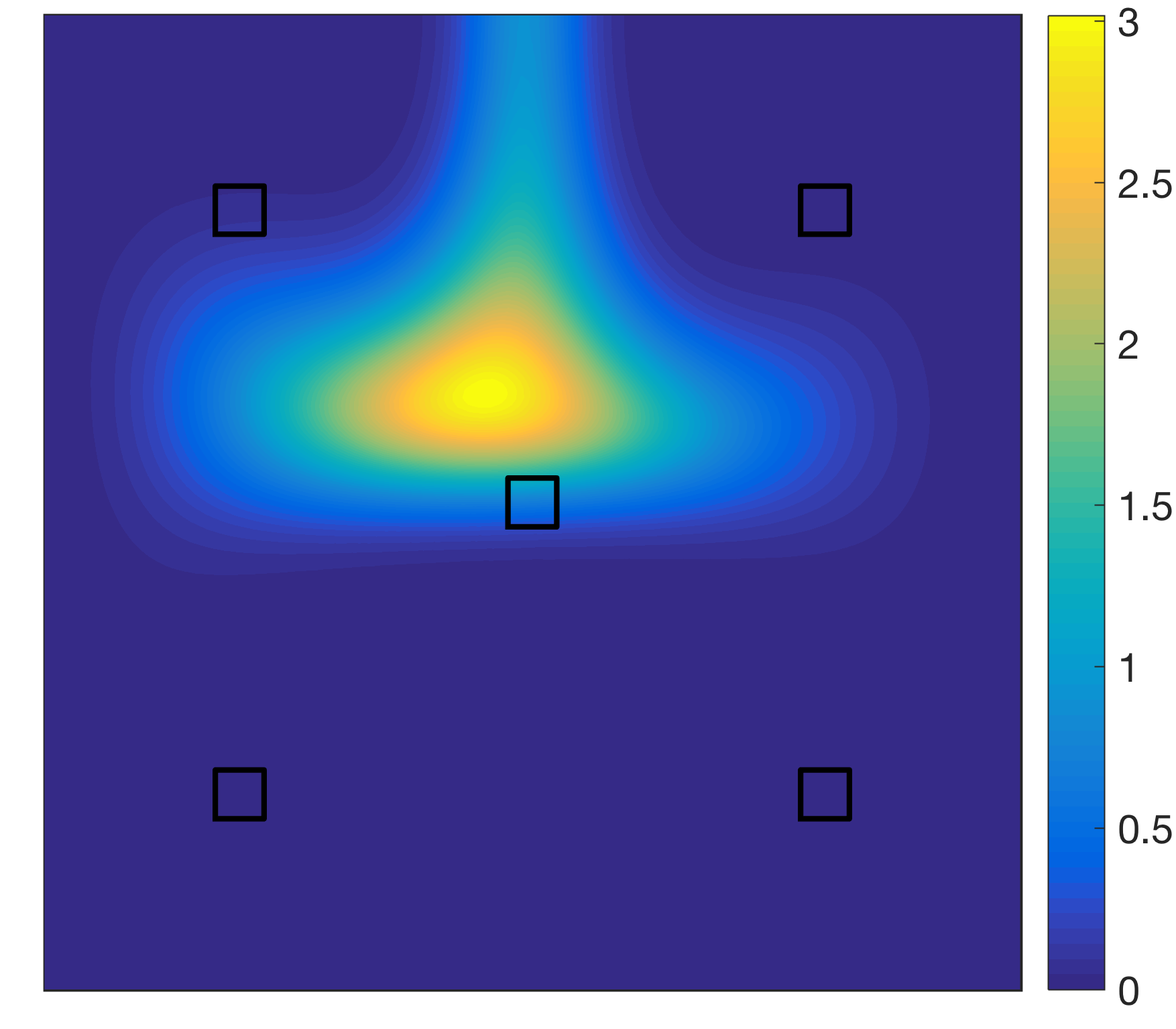} \\
		\raisebox{1.4cm}{\rotatebox[origin=c]{90}{$k=40$}} &
		\hspace*{-0.5cm} \includegraphics[width=0.35\textwidth]{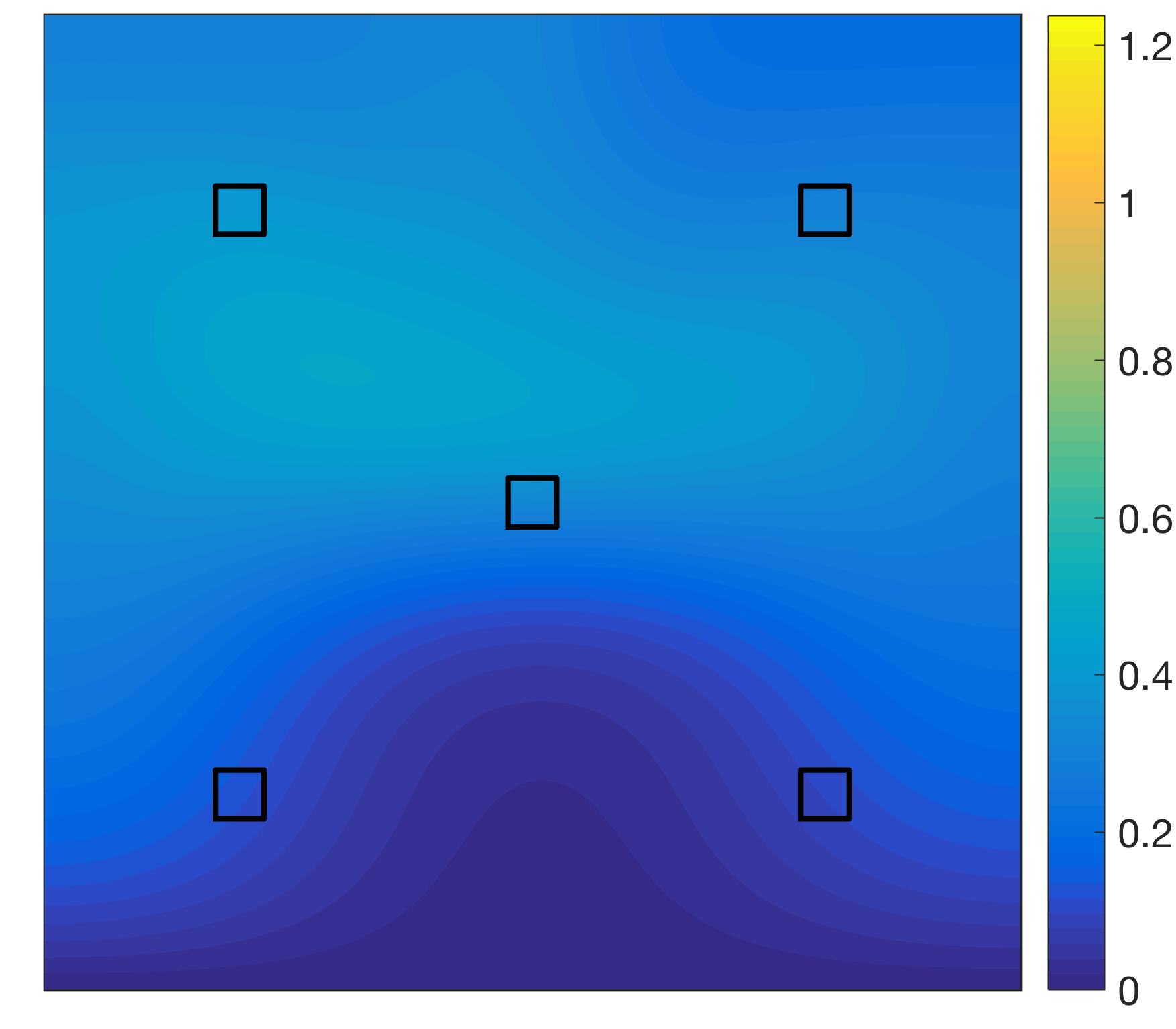} &
		\hspace*{-1cm} \includegraphics[width=0.35\textwidth]{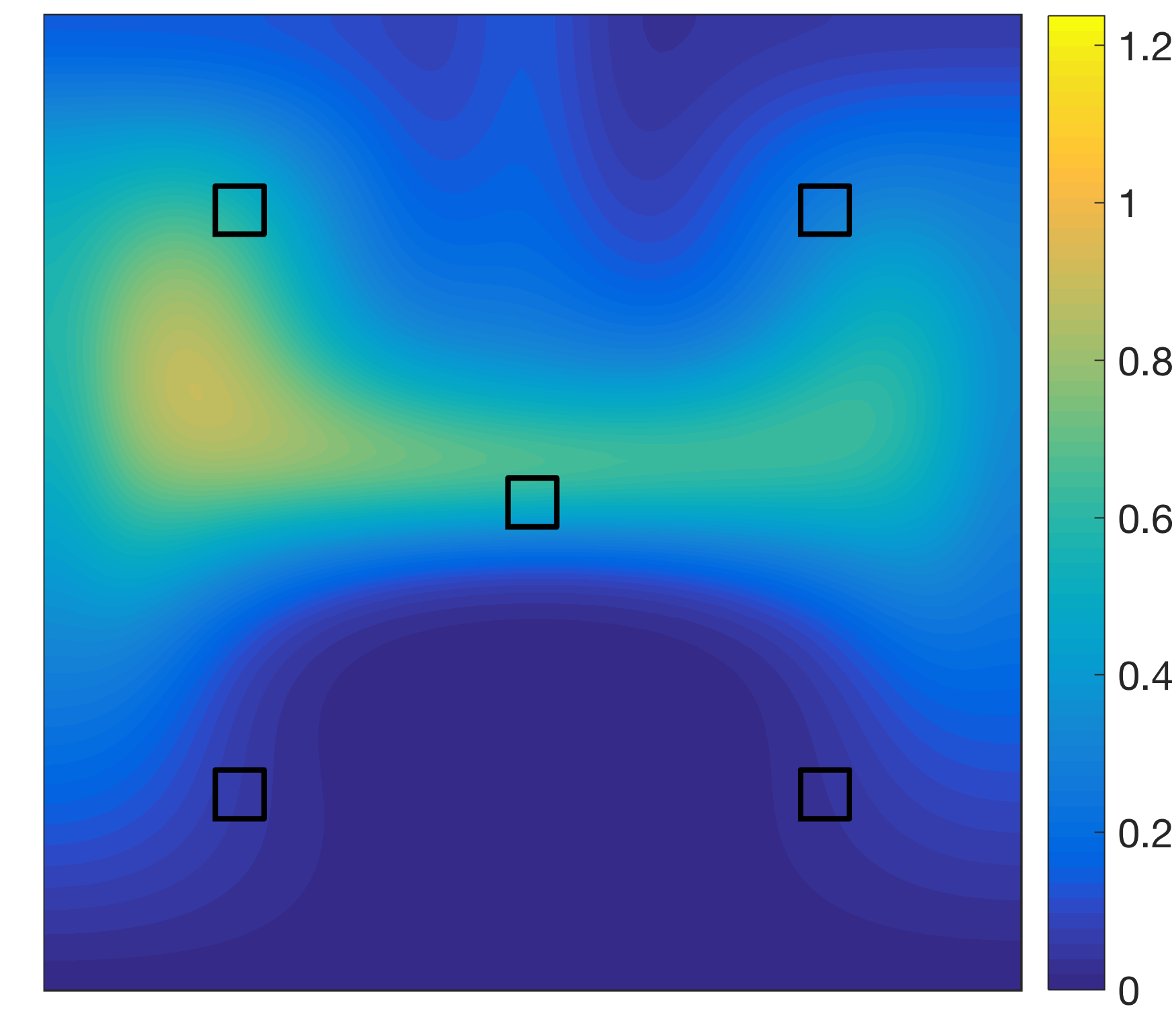} &
		\hspace*{-1cm} \includegraphics[width=0.35\textwidth]{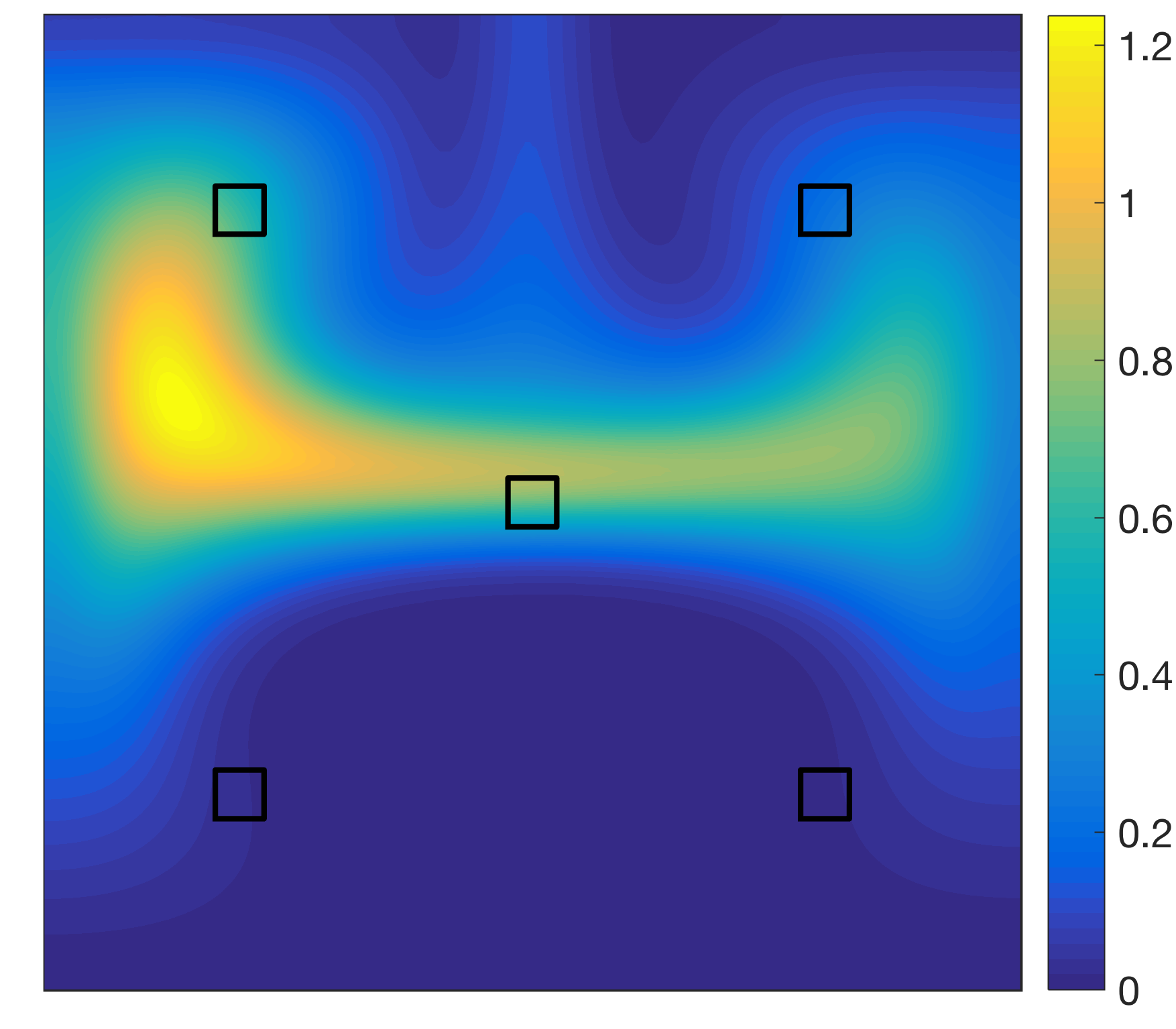} \\
		\raisebox{1.4cm}{\rotatebox[origin=c]{90}{$k=80$}} &
		\hspace*{-0.5cm} \includegraphics[width=0.35\textwidth]{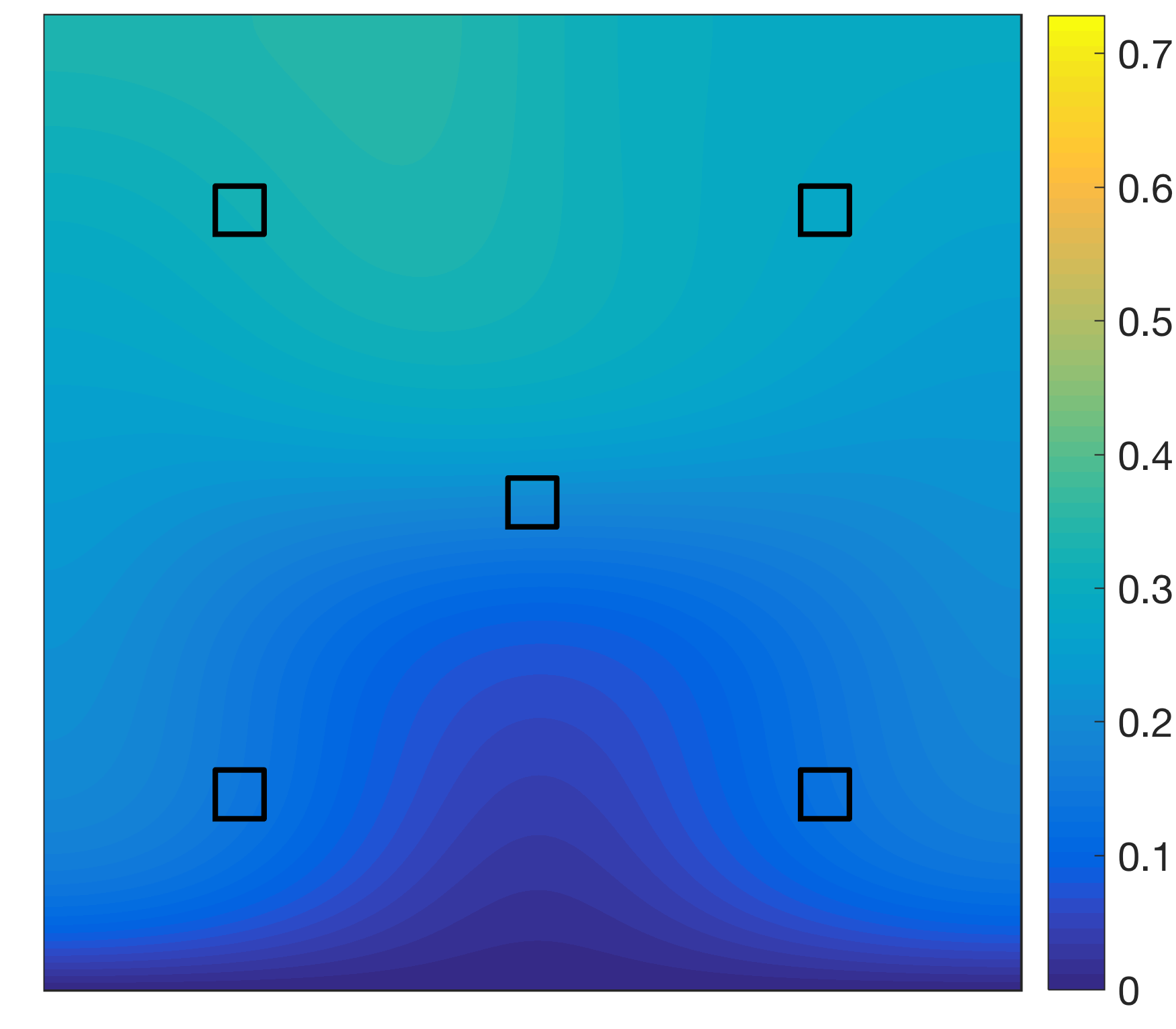} &
		\hspace*{-1cm} \includegraphics[width=0.35\textwidth]{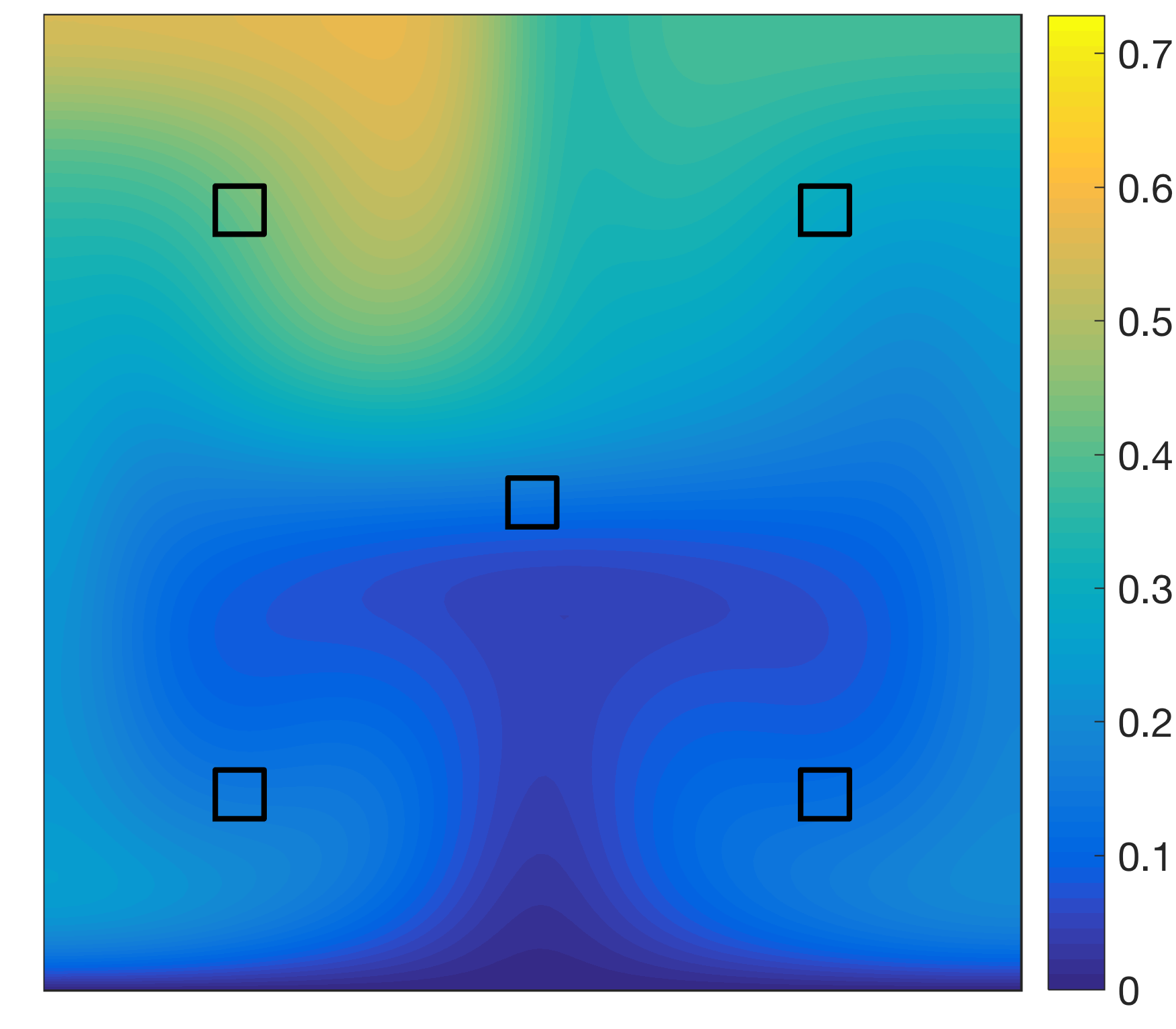} &
		\hspace*{-1cm} \includegraphics[width=0.35\textwidth]{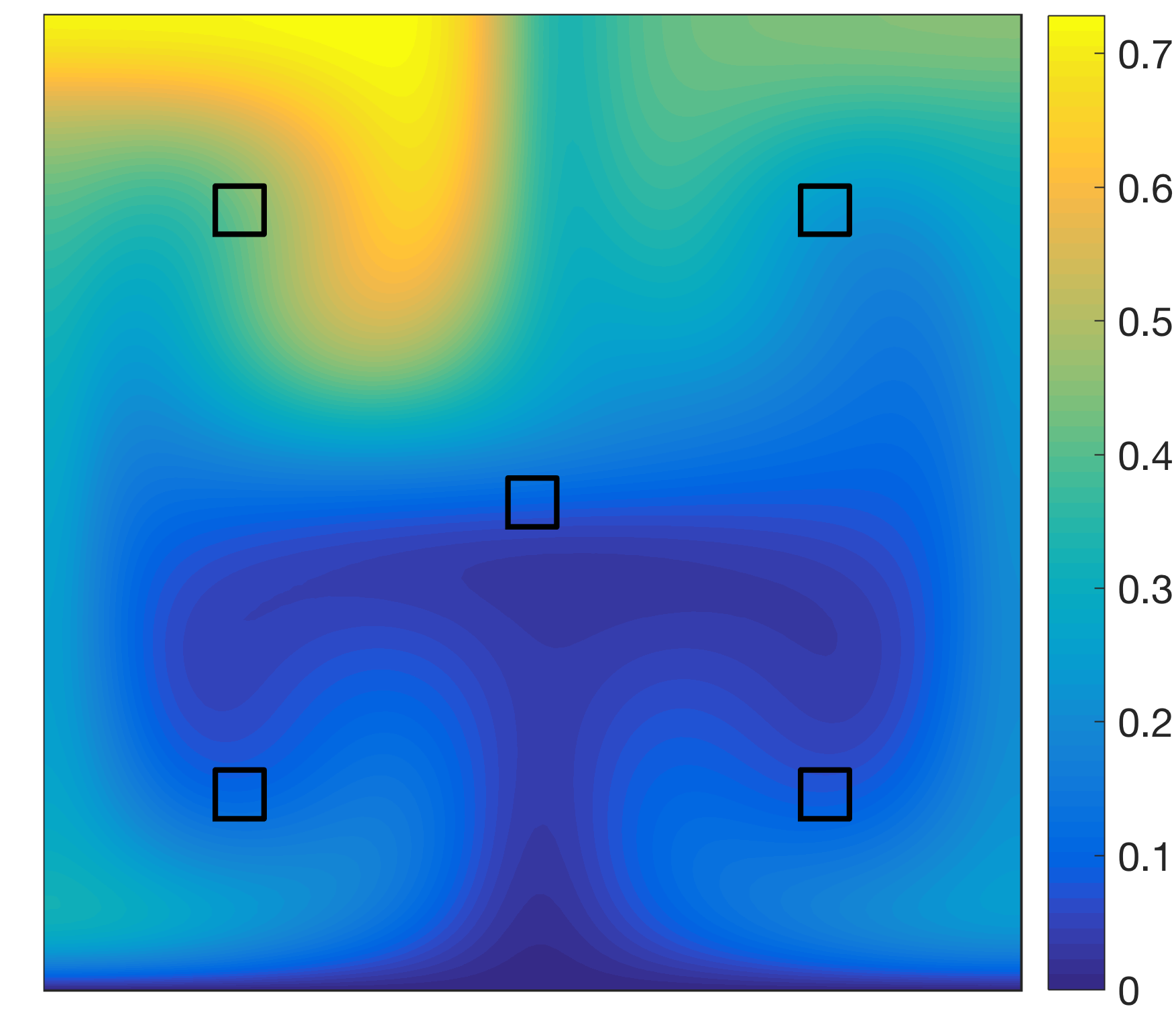} \\
		\raisebox{1.4cm}{\rotatebox[origin=c]{90}{$k=160$}} &
		\hspace*{-0.5cm} \includegraphics[width=0.35\textwidth]{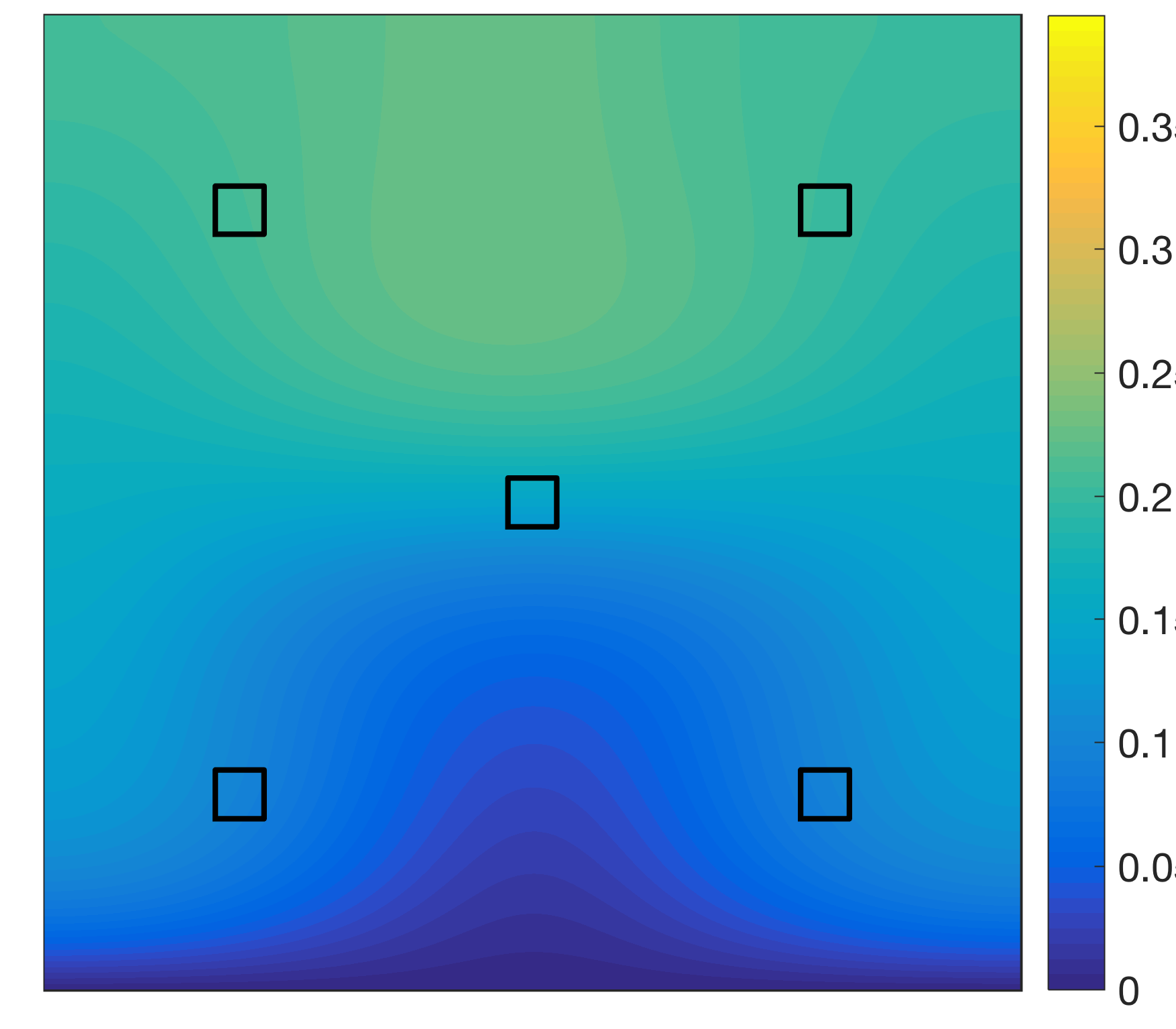} &
		\hspace*{-1cm} \includegraphics[width=0.35\textwidth]{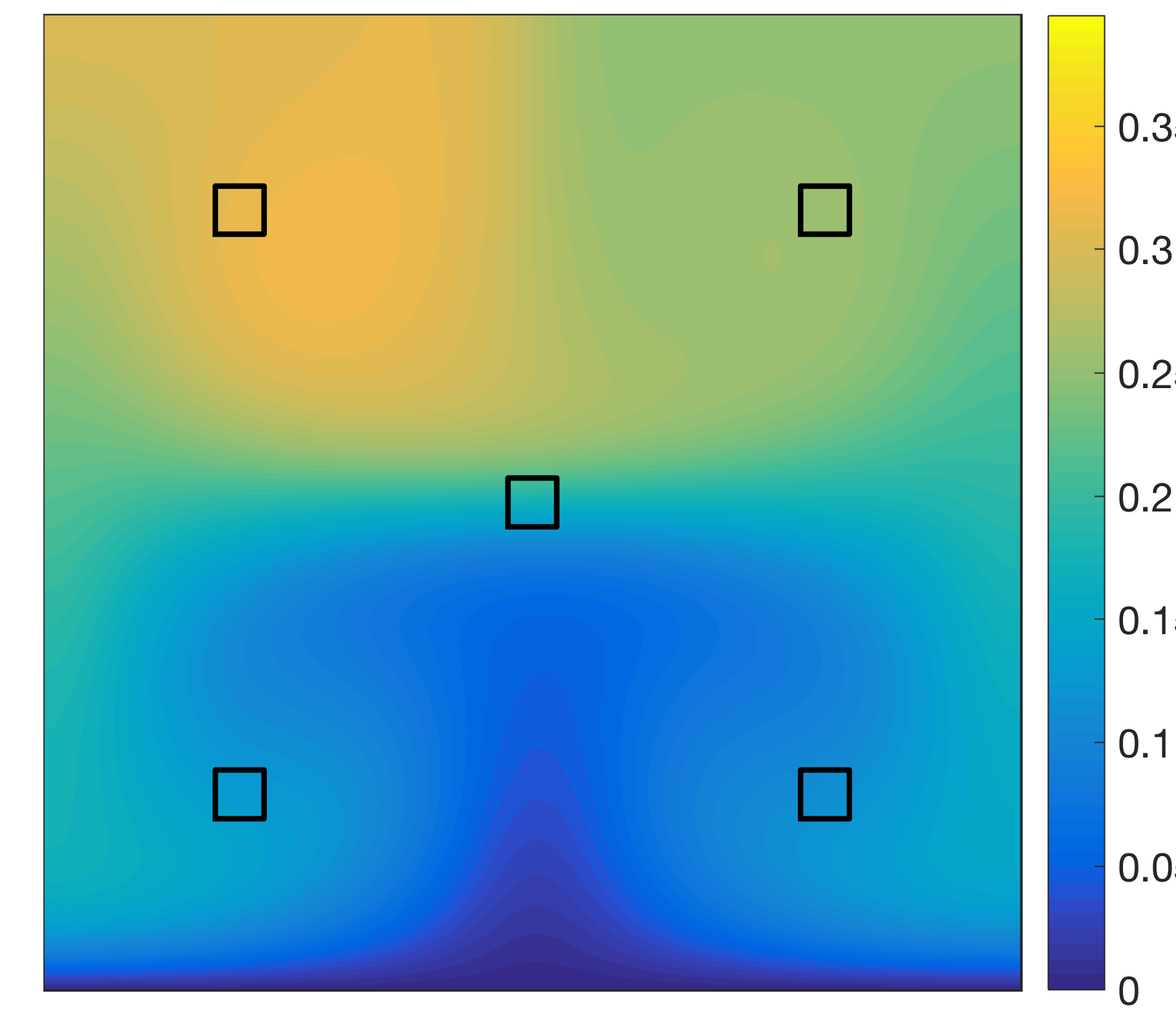} &
		\hspace*{-1cm} \includegraphics[width=0.35\textwidth]{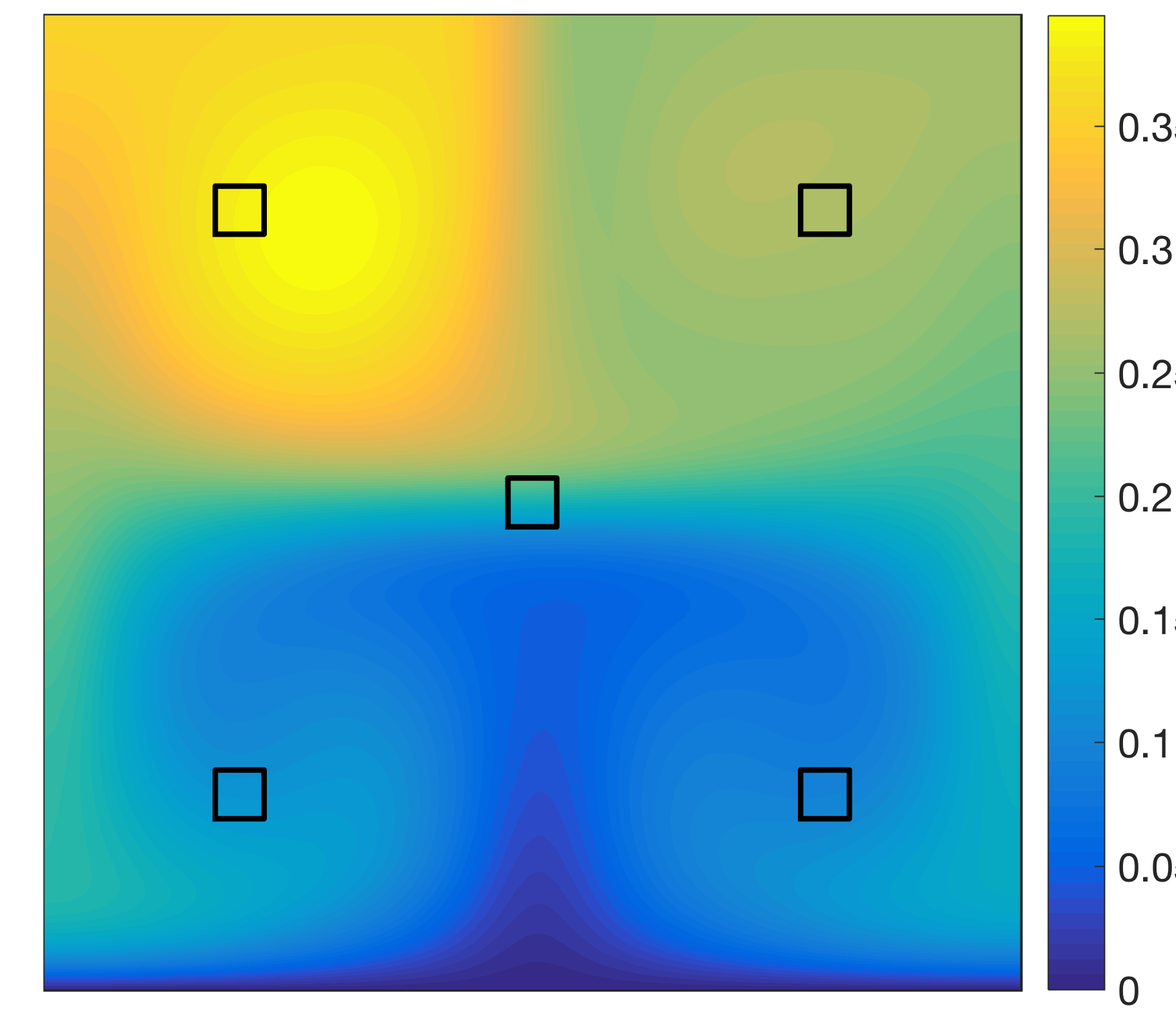}
	\end{tabular}
	\caption{State solution for the true initial condition $y_0^\textrm{true}$ and for three different parameters~$\mu$.}
	\label{fig:state}
\end{figure}

A preconditioned Newton-CG method takes between $30$ seconds for $\mu = 10$ (requiring 31 CG iterations) and $54$ seconds for $\mu = 50$ (requiring 56 CG iterations) to solve the full-order strong-constraint 4D-Var problem. For the weak-constraint case, the solution time ranges from $114$ seconds ($\mu = 10$, 81 CG iterations) to $189$ seconds ($\mu = 50$, 137 CG iterations). In Figure~\ref{fig:state}, we plot the concentration of the pollutant for three different parameter values and various timesteps. The influence of the Taylor-Green vortex and the P\'{e}clet number on the solutions is clearly visible. In Figure~\ref{fig:outputs} on the left, we plot the five true outputs $C y^{k,\textrm{true}}$ over time (the numbering and color of the curves refer to the sketch in Figure~\ref{fig:domain}). The corresponding noisy measurements $z_d^k$ used for the data assimilation are shown on the right. We note that all computations were performed in Matlab on a computer with 2.6 GHz Intel Core i7 processor and 16 GB of RAM.

\begin{figure}[ht]
	\includegraphics[width=0.48\textwidth]{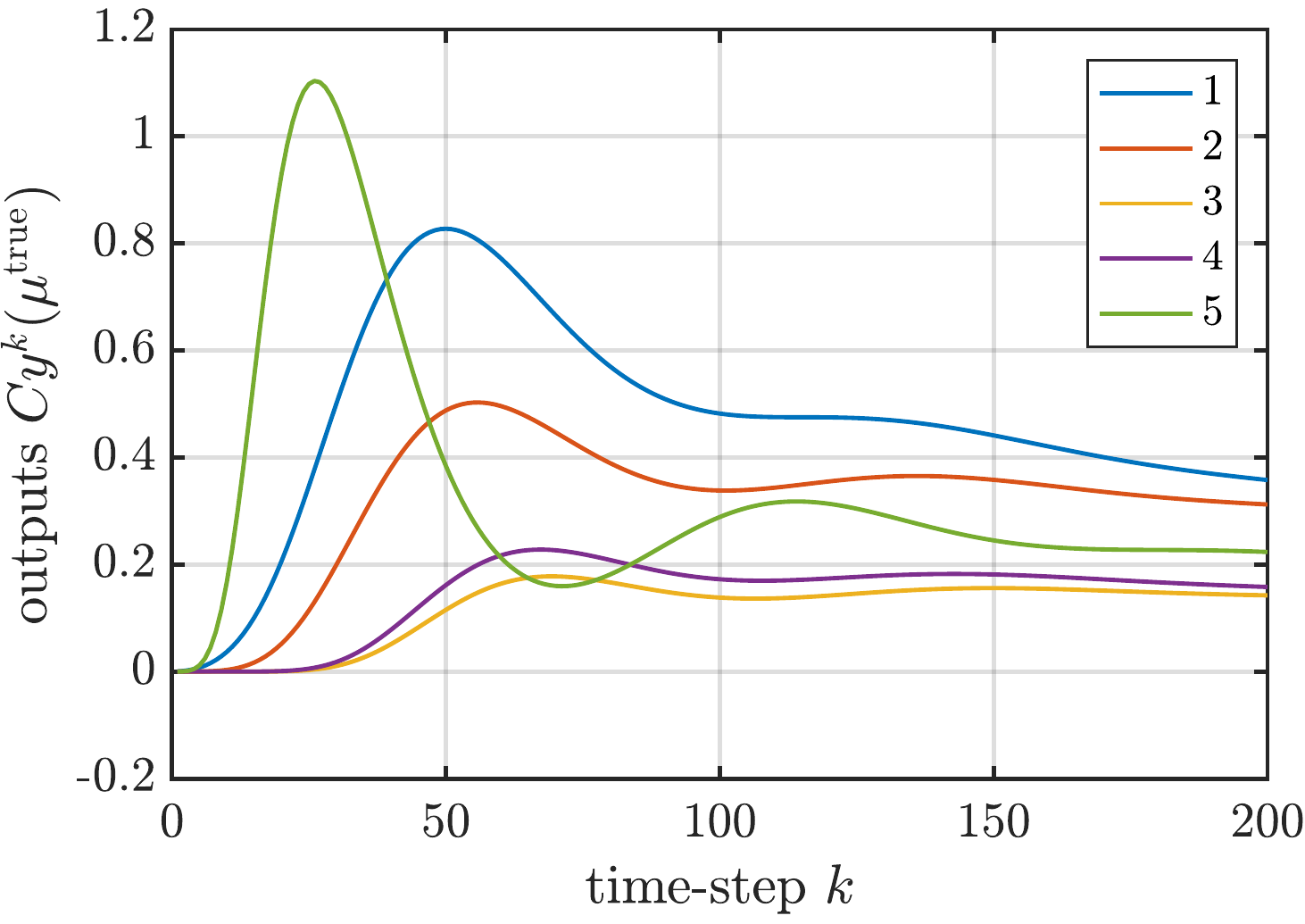} \hfill
	\includegraphics[width=0.48\textwidth]{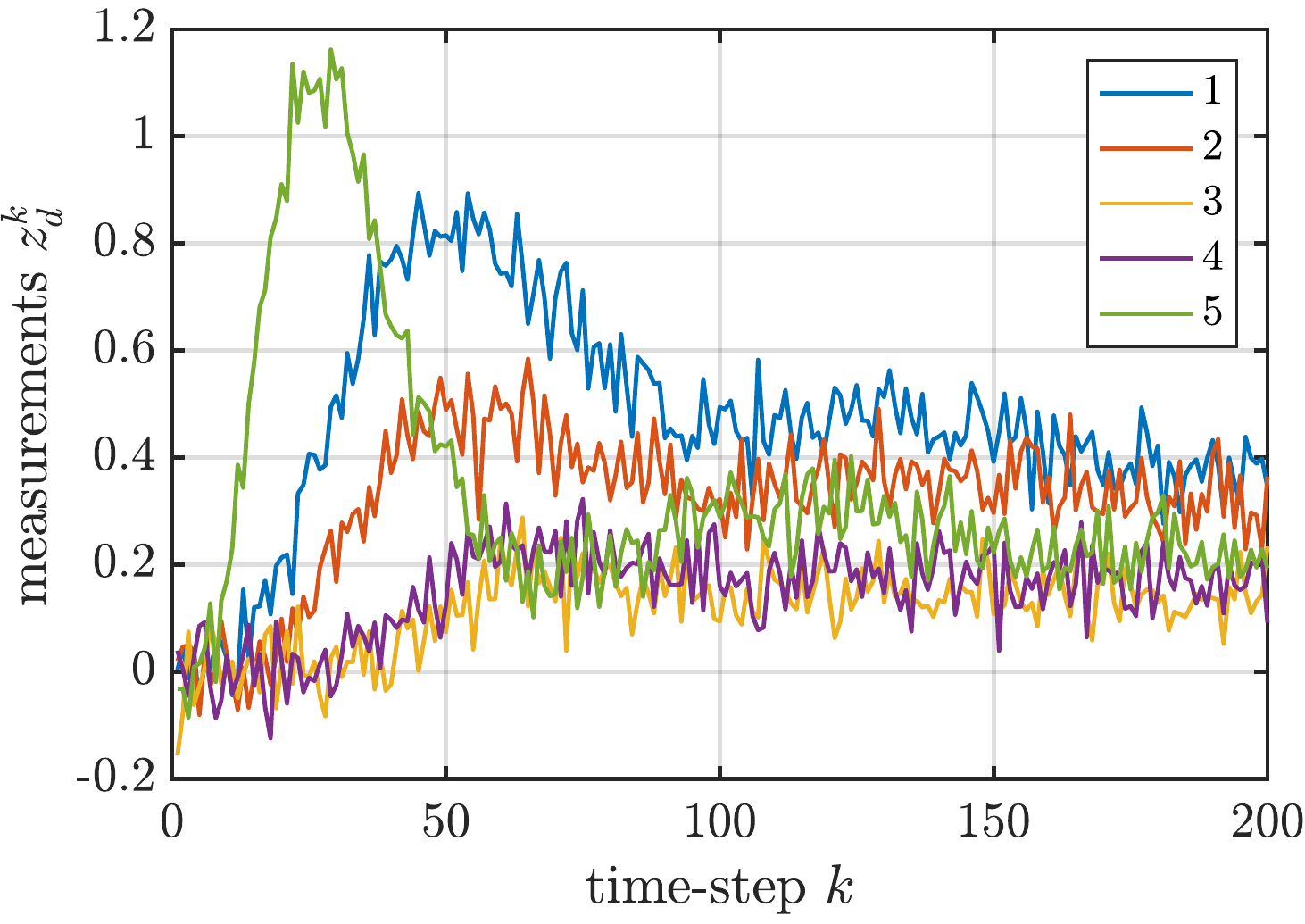}
	\caption{Outputs $C y^k(\mu^\textrm{true})$ and associated noisy output measurements $z_d^k$ over time.}
	\label{fig:outputs}
\end{figure}

\subsection{Reduced-order 4D-Var approach}
\label{ssec:rom4dvar}

We consider the strong- and weak-constraint 4D-Var data assimilation problem separately and present results for the performance of the reduced-order approach for each setting. We thus build different reduced basis spaces for the strong- and weak-constraint case by employing the Greedy sampling procedure described in Section~\ref{ssec:greedy_strong} and \ref{ssec:greedy_weak}, respectively. For both, we choose $\mu^\textrm{start} = 10$ and a training set consisting of 40 equidistant parameters over the parameter domain $\cD$. We also prescribe the number of Greedy iterations to $N_{\max} = 80$ (strong) and $N_{\max} = 100$ (weak) resulting in a relative error bound tolerance of approximately  $10^{-2}$. 

In Figure~\ref{fig:est_err_u_rel} we plot the maximum relative error and error bound over a test sample consisting of 20 randomly chosen parameters in $\cD$ versus the number of Greedy iterations $N$. The relative error and bound are defined as $\norm{u^*(\mu) - u_N^*(\mu)}_U/\norm{u^*(\mu)}_U$ and $\Delta_N^u(\mu)/\norm{u^*(\mu)}_U$ in the strong-constraint case, and by $\Big( \tau \sum_{k=1}^{K} \norm{u^{*,k}(\mu) - \uNoptk(\mu)}_U^2 \Big)^{1/2} /\Big( \tau \sum_{k=1}^{K} \norm{u^{*,k}(\mu)}_U^2 \Big)^{1/2}$ and $\tilde{\Delta}_N^u(\mu)/ \Big( \tau \sum_{k=1}^{K} \norm{u^{*,k}(\mu)}_U^2 \Big)^{1/2}$ in the weak-constraint case. We observe that the error and bound converge at the same rate and that the effectivities, i.e., the ratio of the bound and the error, thus remain almost constant over $N$. The mean effectivities over the test sample for $N_{\rm max}$ are $480$ in the strong-constraint case and $40$ in the weak-constraint case. We note that maximum dimensions of the reduced basis state/adjoint and control spaces are $N_{Y,{\rm max}} = 2 N_{\rm max} = 160$ and $N_{U,{\rm max}}^0 = 21$ (strong-constraint), and $N_{Y,{\rm max}} = 2 N_{\rm max} + 1 = 201$ and $N_{U,{\rm max}} = N_{\rmmax} = 100$ (weak-constraint). Especially in the strong-constraint case, we thus obtain a considerable reduction in the dimension of the control space from $\cN = {13,131}$ to $N_{U,{\rm max}}^0 = 21$. This will also be reflected in the required number of CG iterations to solve the reduced-order 4D-Var problem (see below).


\begin{figure}[ht]
	\includegraphics[width=0.48\textwidth]{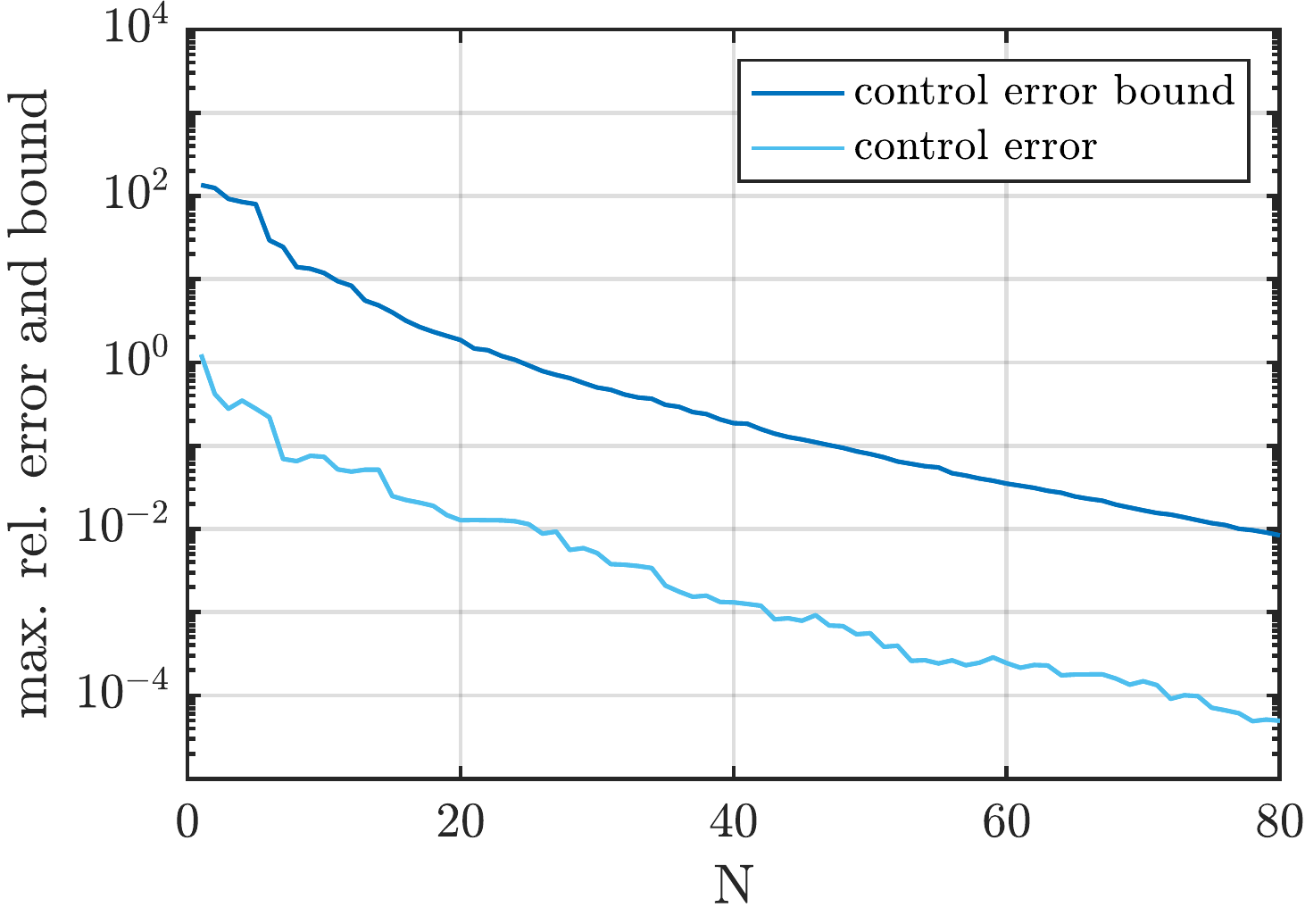} \hfill
	\includegraphics[width=0.48\textwidth]{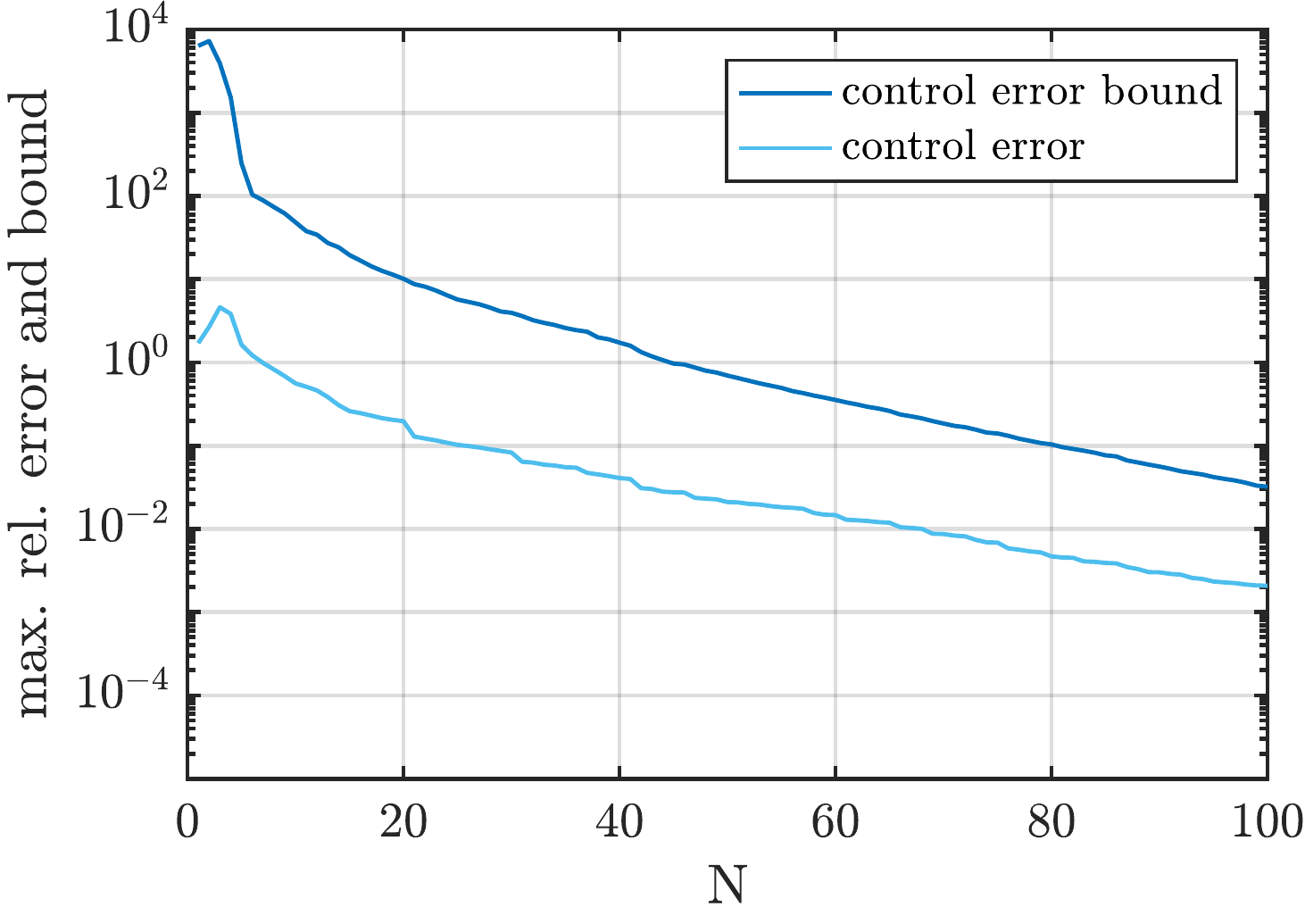}
	\caption{Maximum relative control error and error bound over number of Greedy iterations~$N$ for strong-constraint case (left) and weak-constraint case (right).}
	\label{fig:est_err_u_rel}
\end{figure}

We next report on the online computational times of our reduced-order approach. Similar to the full-order approach, the reduced-order solution times also depend on $\mu$ (smaller for $\mu = 10$ and higher for $\mu = 50$) and of course also strongly on $N$. We first consider the strong-constraint case:  the solution times for the reduced-order 4D-Var problem range from $10$ milliseconds to $1.37$ seconds, the evaluation of the \textit{a~posteriori} error bound $\Delta_N^u(\mu)$ takes between $2.8$ and $29$ milliseconds. We note that the computation of the error bound is much faster than the solution of the 4D-Var problem itself. Furthermore, we note that the computational time to evaluate the error bound only depends on $N$ and not on $\mu$ (i.e., evaluating the bound for fixed $N$ at $\mu = 10$ or $\mu = 50$ takes the same time). The overall online speed-up for $N = N_{\rm max}$ thus ranges from approximately $23$ to $40$. 

In the weak-constraint case, the solution times for the reduced-order 4D-Var problem range from $99$ milliseconds to $12.6$ seconds, the evaluation of the \textit{a~posteriori} error bound $\tilde{\Delta}_N^u(\mu)$ takes between $4.8$ and $71$ milliseconds. Again, the evaluation of the error bound is much faster than the solution of the 4D-Var problem itself. The online speed-up for $N = N_{\rm max}$ is now approximately $15$. 

\cbb{In order to illustrate the connection between the approximation error and the online solution time, we plot the average online solution time of the reduced-order 4D-Var problem versus the average relative error over the test sample in Figure~\ref{fig:comp_times}. Recall that the full-order solution takes approximately $30-54$ seconds for the strong-constraint case and $114-189$ seconds for the weak-constraint case.

\begin{figure}[ht]
	\includegraphics[width=0.48\textwidth]{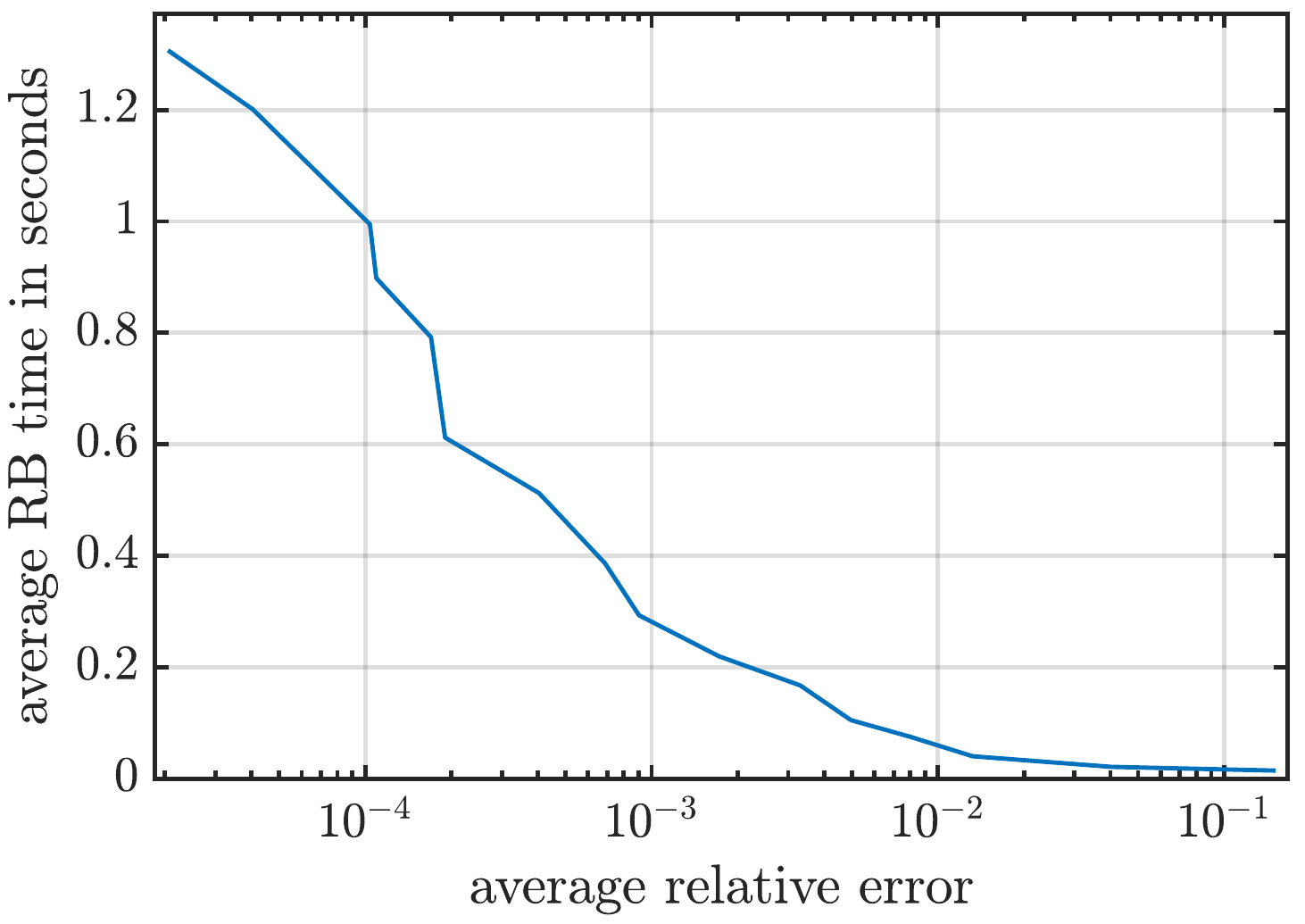} \hfill
	\includegraphics[width=0.48\textwidth]{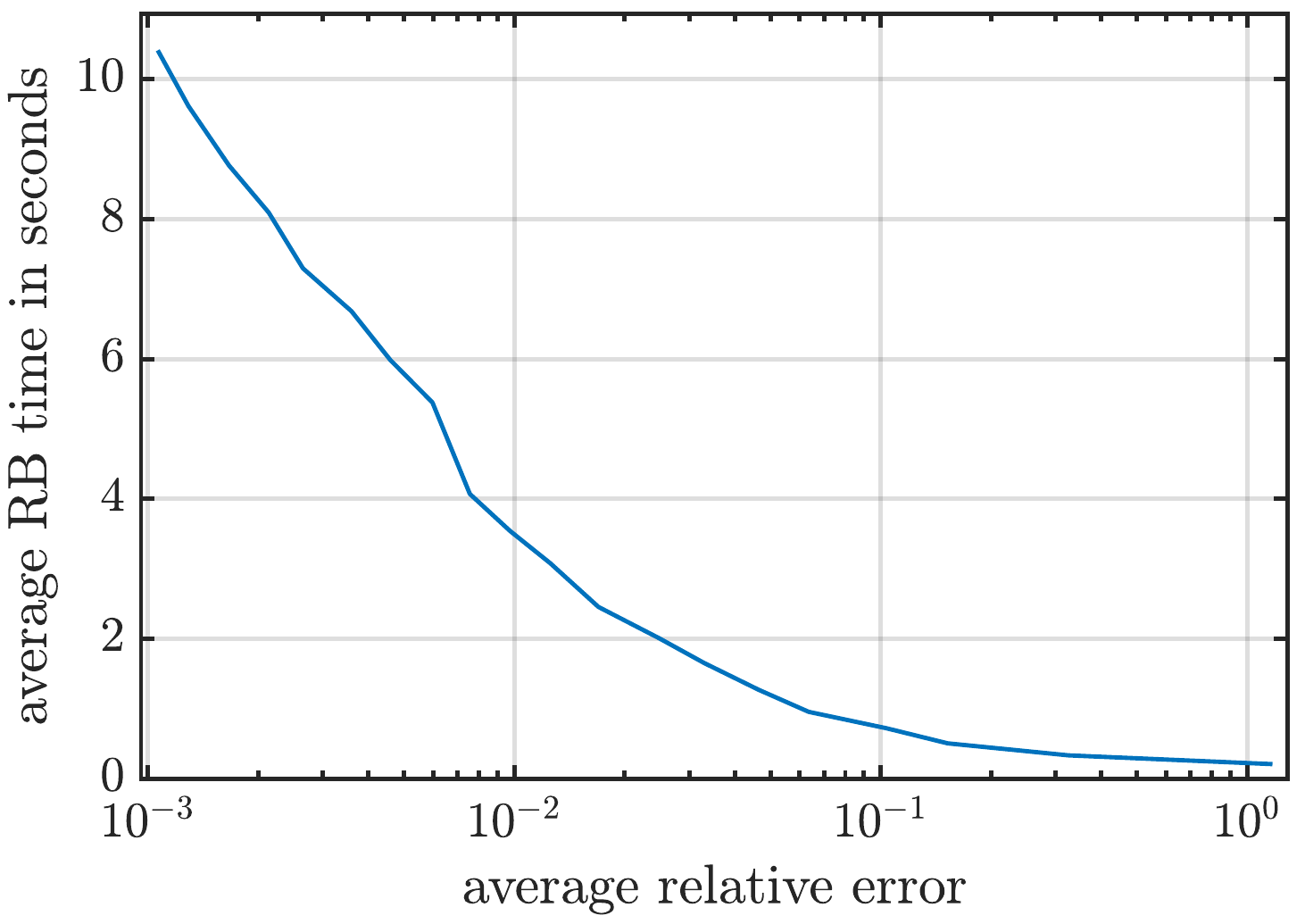}
	\caption{Average online solution time of the reduced-order 4D-Var problem over the average relative error for strong-constraint case (left) and weak-constraint case (right).}
	\label{fig:comp_times}
\end{figure}
}

%
%
%

%

We next show results for the number of CG iterations required to solve the reduced-order 4D-Var problem. In Figure~\ref{fig:cg_iters}, we plot the number of CG iterations as a function of the parameter $\mu$ for various values of $N$ and $N_U$ on the left for the strong-constraint case and on the right for the weak-constraint case. In the same plots, we also show the number of CG iterations required to solve the full-order problem. We observe a different behavior in the strong- and weak-constraint case. We first note that in the weak-constraint case the number of reduced-order CG iterations converges to the number of full-order CG iterations with increasing $N$. However, in the strong-constraint case the number of reduced-order CG iterations is bounded by $N_U^0$, which is significantly smaller than $N$. The number of reduced-order CG iterations are thus almost constant over $\mu$ for given $N$ and are considerably smaller than the number of full-order CG iterations even for $N = N_{\rm max}$.

\begin{figure}[ht]
	\includegraphics[width=0.48\textwidth]{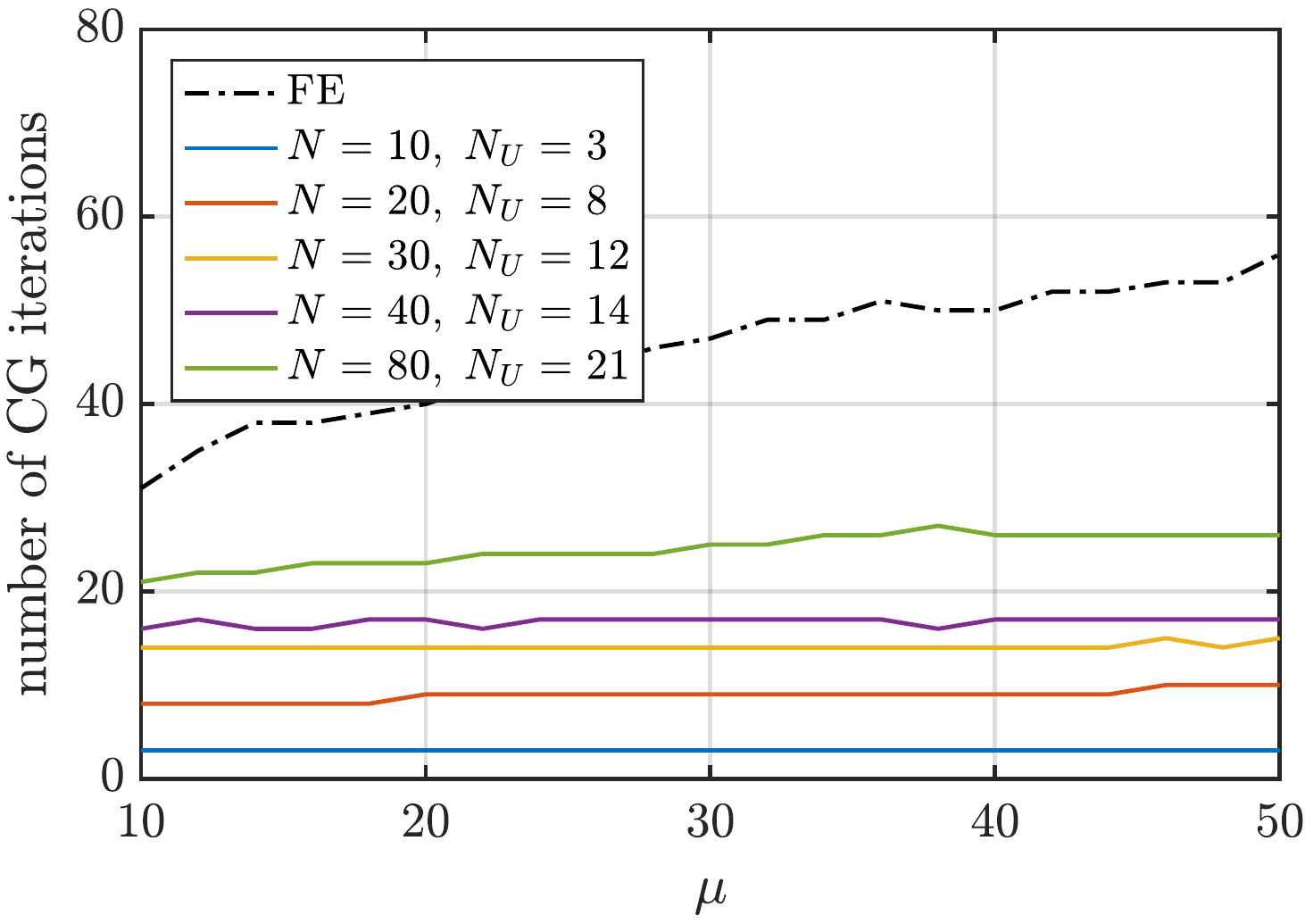} \hfill
	\includegraphics[width=0.48\textwidth]{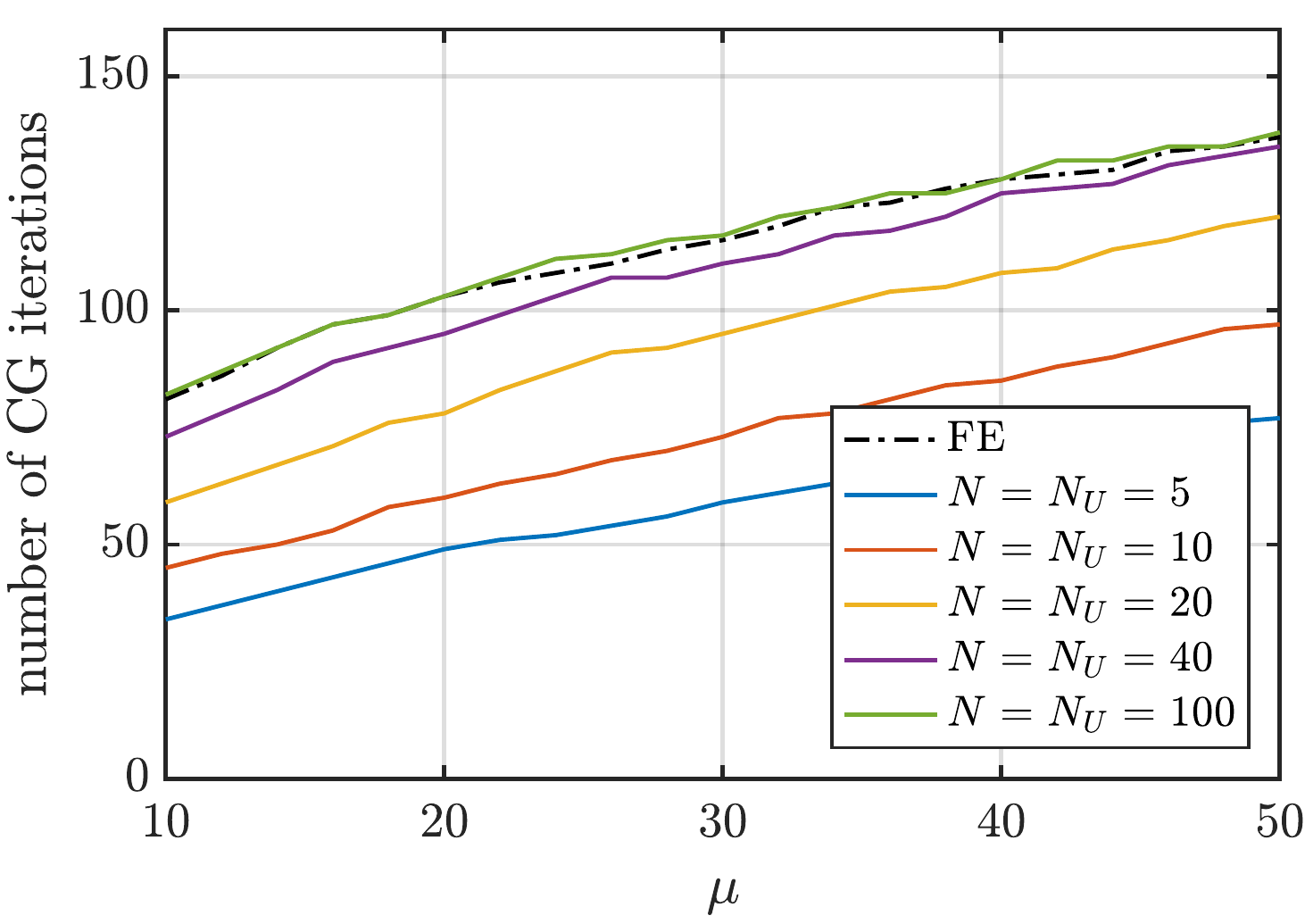}
	\caption{Required number of CG iterations for solving the full- and reduced-order 4D-Var problem in dependence of the parameter $\mu$ and the number of Greedy iterations $N$. Strong-constraint case (left) and weak-constraint case (right).}
	\label{fig:cg_iters}
\end{figure}

Finally, we consider the outer minimization problem and try to estimate the unknown true parameter $\mu^\textrm{true} = 30$ which lead to the noisy measurements. To this end, we define the ``optimal'' parameters $\mu^*$ and $\mu_N^*$ which minimize the full-order and reduced-order cost functionals 
\begin{equation}
	\mu^* = \argmin_{\mu \in \cD} J^*(\mu) \quad \text{and} \quad
	\mu_N^*= \argmin_{\mu \in \cD} J_N^*(\mu),
\end{equation}
respectively.  We compute the optimal estimated parameters $\mu^*$ and $\mu_N^*$ using the Matlab routine \texttt{fminbnd}, which only needs evaluations of the full-order and reduced-order cost functional. We also define the maximum relative cost functional error $e_{J,N}^{\max} = \max_{\mu \in \cD} |J^*(\mu) - J_N^*(\mu)|/|J^*(\mu)|$ and parameter error $e_{\mu,N} := |\mu^* - \mu_N^*|/|\mu^*|$. We present these errors for the strong- and weak-constraint case as a function of $N$ in Table~\ref{tab:err_J_err_mu}. We observe that in both cases the cost functional error and parameter error converge very fast, i.e., the reduced-order approach allows us to recover the optimal parameter $\mu^*$. We also note that the (full-order) optimal parameter is close to the true parameter in the strong-constraint case ($\mu^* = 29.67$ vs.\ $\mu^\textrm{true} = 30$), but that this is not true in the weak-constraint case ($\mu^* = 45.36$ vs.\ $\mu^\textrm{true} = 30$). Since $\mu_N^* \to \mu^*$ with increasing $N$, this is of course also true for --- and the best we can expect of --- the reduced-order optimal parameters.

\begin{table}
\caption{Error in cost functional and estimated parameter over number of Greedy iterations~$N$. Note that $\mu^* = 29.67$ (strong) and $\mu^* = 45.36$ (weak).}
\label{tab:err_J_err_mu}
\begin{center}
\begin{tabular}{rllll}
\hline\noalign{\smallskip}
N & $e_{J,N}^{\max}$ (strong) & $e_{\mu,N}$ (strong) & $e_{J,N}^{\max}$ (weak) & $e_{\mu,N}$ (weak) \\
\noalign{\smallskip}\hline\noalign{\smallskip}
 10  &  3.12e-01  &  4.18e-01  &  2.44e-01  &  6.02e-02 \\
 20  &  7.36e-03  &  1.30e-01  &  1.70e-02  &  9.33e-03 \\
 30  &  8.22e-04  &  1.42e-03  &  3.51e-03  &  1.70e-04 \\
 40  &  1.24e-04  &  4.99e-04  &  6.37e-04  &  3.26e-04 \\
 50  &  1.14e-05  &  2.98e-05  &  2.05e-04  &  3.53e-05 \\
 60  &  4.36e-06  &  1.27e-05  &  9.70e-05  &  3.90e-05 \\
 70  &  3.92e-07  &  4.18e-06  &  3.58e-05  &  1.93e-05 \\
 80  &  8.76e-08  &  9.71e-08  &  1.05e-05  &  4.12e-06 \\
 90  &  -         &  -         &  4.17e-06  &  2.51e-06 \\
100  &  -         &  -         &  1.94e-06  &  3.09e-06 \\
\noalign{\smallskip}\hline
\end{tabular}
\end{center}
\end{table}

\section{Conclusion}
\label{sec:conc}

In this paper, we considered the strong- and weak-constraint 4D-Var data assimilation problem. We presented a reduced-order approach to the 4D-Var problem based on the reduced basis method and proposed rigorous and efficiently evaluable \textit{a~posteriori} error bounds for the optimal control, i.e., the initial condition in the strong-constraint setting and the model-error forcing in the weak-constraint setting. For both instances we showed numerical results confirming the validity of the proposed approach. We also presented theoretical results for the combined case with unknown initial condition \textit{and} model-error forcing. 

We note that although we consider a parametrized problem here, the error bounds can also be used in the non-parametrized reduced-order setting and are independent of how the reduced-order spaces are constructed. The bound thus directly applies to reduced-order approaches where the spaces are constructed, e.g., using empirical orthogonal functions, POD, or dual-weighted POD~\cite{DN2008}. We also believe that the error bounds can be gainfully applied in a multi-fidelity approach to solve the 4D-Var problem, e.g., in a trust-region approach as proposed in~\cite{CNF2011,DNZ+2013}. 

Although we also presented results for the error in the cost functional and for estimating the unknown model parameter, we currently cannot provide rigorous and sharp \textit{a~posteriori} error bounds for these quantities. \cbb{Furthermore, we only considered a fixed setting for the noise level and regularization parameter here, a detailed analysis of the influence of these parameters on the performance of the reduced order model has not been performed.}  These are topics of current and future research in our groups.

\cbb{
\appendix
\section{Continuous 4D-Var Formulation}
\label{sec:cont}

The strong-constraint 4D-Var problem for a linear parabolic PDE on $(0,T)\times\Omega$, with $\Omega$ a Lipschitz domain, 
\begin{equation}
\label{eq0}
\partial_ty + Ay = f, \quad y|_{\partial\Omega}=0, \qquad y(0)=u,
\end{equation}
classically rewrites as the optimal control problem:
$$
\text{Find } (y^*,u^*) \in \mathop{\rm arginf}_{ (y,u)\in\mathcal{Y}\times U 
\text{ satisfies \eqref{eq0} } } J
$$
with a lower semi-continuous cost functional 
\begin{equation}
\label{cost_strong}
 J = \frac\lambda2 \norm{u-u_d}^2_U + \frac12 \int_0^T \norm{Cy-z_d}^2_D
\end{equation}
on the 
tensor-product of
$ \mathcal{Y} 
:= \{y\in L^2(0,T;H^1_0(\Omega)); \partial_t y\in L^2(0,T;H^{-1}(\Omega))\}$
and 
$U 
:= L^2(\Omega)$.
If the 
observation operator $C$ has a unique continuation in $\mathcal{Y}$,
$J$ is \emph{coercive} and strictly convex.
Then, if $f\in L^2((0,T)\times\Omega)$ so the set of admissible states is non-empty, there exists a unique solution, see e.g.~\cite{Fursikov2000}.
To characterize and compute the solution, one can use duality techniques following~\cite{PBG+1964} or~\cite{ET1976}.
%
On introducing a Lagrange multiplier $(p^*,v^*)\in \mathcal{Y}\times U$, $p^*(T)= 0$ for the constraint, it is classical that the solution should satisfy 
\cite{MS2002}
\begin{subequations}
\begin{align}
(u^*,\psi)_U-(\lambda v^*,\psi)_U & =(u_d,\psi)_U \ \forall \psi\in U
\\
(p^*(0),\varphi)_Y - (v^*,\varphi)_U & = 0
\\
(Cy^*,C\varphi)_D-(\partial_tp^*-A^Tp^*,\varphi)_Y & =(Cz_d,C\varphi)_D \ \forall \varphi\in \mathcal{Y}
\\
(y^*(0),\phi)_Y - (u^*,\phi)_U& = 0
\\
(\partial_ty^* + Ay^*,\phi)_Y & = (f,\phi)_Y \ \forall \phi \in \mathcal{Y}
\end{align}
\end{subequations}
which is a well-posed saddle-point problem, well-approximated by the discretization~\refeq{OS_fe_strong} \cite{EG2010}, again on the condition that the observation operator $C$ has a unique continuation in $\mathcal{Y}$.
Note that first adequately discretizing $J$  then leads to exactly the same discrete Euler-Lagrange equations as~\refeq{OS_fe_strong}.

The weak-constraint 4D-Var problem is also classical, see e.g. \cite{Fursikov2000}. 
The optimal control problem becomes
$$
\text{Find } (y^*,u^*) \in \mathop{\rm arginf}_{ (y,u)\in\mathcal{Y}_{y_0}\times\mathcal{U}
\text{ satisfies \eqref{eq1_seb} } } J
$$
for
\begin{equation}
\label{eq1_seb}
\partial_ty + Ay = f + Bu, \quad y|_{\partial\Omega}=0, \qquad y(0)=y_0,
\end{equation}
with the lower semi-continuous cost functional 
\begin{equation}
\label{cost_weak}
 J = \frac12 \int_0^T \norm{u-u_d}^2_U + \frac12 \int_0^T \norm{Cy-z_d}^2_D
\end{equation}
on the 
tensor-product of
$ \mathcal{Y}_{y_0} 
:= \{y\in L^2(0,T;H^1_0(\Omega)); \partial_t y\in L^2(0,T;H^{-1}(\Omega)); y(0)=y_0 \}$
and 
$\mathcal{U}:= L^2((0,T)\times\Omega)$;
the saddle-point becomes
\begin{subequations}
\begin{align}
(u^*,\psi)_U-(B^Tp^*,\psi)_Y & =(u_d,\psi)_U \ \forall \psi\in U
\\
(Cy^*,C\varphi)_D-(\partial_tp^*-A^Tp^*,\varphi)_Y & =(Cz_d,C\varphi)_D \ \forall \varphi\in \mathcal{Y}
\\
(\partial_ty^* + Ay^*,\phi)_Y  - (Bu^*,\phi)_Y & = (f,\phi)_Y \ \forall \phi \in \mathcal{Y}
\end{align}
\end{subequations}
while existence and uniqueness of a soluion still hold under the same conditions.
}


\bibliographystyle{spmpsci}      
\bibliography{../../Tex/bibtex/bib/All_References}

\end{document}